\documentclass[preprint]{elsarticle}
\usepackage{amsmath}
\usepackage{amsthm,nccmath}
\usepackage{amssymb}
\usepackage{lipsum} 
\usepackage{graphicx}
\graphicspath{ {./images/} }
\newcommand{\szero}{\mathbf{0}}
\newcommand{\sym}{\operatorname{sym}}

\newtheorem{theorem}{Theorem}[section]
\newtheorem{corollary}{Corollary}[theorem]

\theoremstyle{definition}
\newtheorem{definition}{Definition}[section]
\newtheorem{proposition}{Proposition}[section]
\newtheorem{example}{Example}[section]
\newtheorem{remark}{Remark}[section]

\newtheorem{protocol}{Protocol}
\newtheorem{algorithm}{Algorithm}

\usepackage{hyperref}
\usepackage{enumitem}
\usepackage{algorithm}
\usepackage{multirow}
\usepackage{tabularx}
\usepackage{caption}
\usepackage{appendix}
\usepackage{comment}
\usepackage[noend]{algpseudocode}
\hypersetup{
    colorlinks=true,
    linkcolor=black,
    citecolor=black,
    urlcolor=black
}
\begin{document}

\title{Solving one-sided linear systems over symmetrized and supertropical semirings}

\author[rvt]{Sulaiman Alhussaini\corref{cor1}}
\ead{saa399@student.bham.ac.uk}

\author[rvt]{Serge{\u\i} Sergeev}
\ead{s.sergeev@bham.ac.uk}

\address[rvt]{University of Birmingham, School of Mathematics, Edgbaston B15 2TT}

\cortext[cor1]{Corresponding author.  email: saa399@student.bham.ac.uk}

\begin{abstract}
 One-sided linear systems of the form ``$Ax=b$'' are well-known and extensively studied over the tropical (max-plus) semiring and wide classes of related idempotent semirings. The usual approach is to first find the greatest solution to such system in polynomial time and then to solve a much harder problem of finding all minimal solutions. We develop an extension of this approach to the same systems over two well-known extensions of the tropical semiring: symmetrized and supertropical, and discuss the implications of our findings for the tropical cryptography.  
\end{abstract}



\maketitle

\section{Introduction}
The symmetrized tropical semiring $\mathbb{S}$ is an extension of the tropical semiring $\mathbb{R}_{\max}$, similar to how the integers $\mathbb{Z}$ extend the natural numbers $\mathbb{N}$. This extension introduces the notion of signs (positive, negative, or balanced) into the framework of tropical algebra. Various algebraic problems over symmetrized tropical semiring are extensively treated by Baccelli et al.~\cite{baccelli_book} and Gaubert~\cite{GaubertThesis}, where the formulation of a tropical version of Cramer's formulas using the symmetrized semiring is one of the most notable results. A related but different extension is the supertropical semiring introduced by Izhakian~\cite{izhakian_supertropical}, which instead distinguishes tangible and ghost elements and uses the ghost surpass relation to recover some classical linear algebra features. Both extensions can be seen as special cases of the tropical layered semiring discussed in Akian, Gaubert and Guterman~\cite{AkianGaubertGuterman2014}. Note that the linear algebra over a more general structure called semiring pairs has been developed more recently in Akian, Gaubert and Rowen~\cite{AkianGaubertRowen2023}.

Solving the one-sided system $A \otimes x=b$  over the tropical semiring is a ``classical'' problem of the tropical matrix algebra, which was first studied in 1960's by Cuninghame-Green, see the monograph~\cite{Cuninghame-Green1979} and references therein, and Vorob'yev~\cite{vorobyev1967extremal}. They gave a formula for the greatest candidate solution of this system, but the minimal solutions were systematically studied much later by, e.g., Markovskii~\cite{Markovskii2004}, and over more general semirings by Di Nola et al.~\cite{DiNola87}. The relation between the problem of finding minimal solutions of $A\otimes x=b$ with the minimal hypergraph transversal problem was discussed by Elbassioni~\cite{Elbassioni}.

 In the tropical semiring, determining the solvability of the system $A\otimes x=b$ is straightforward, and the entire solution set can be fully described by computing a unique greatest solution, which is easily obtained using an explicit formula. Furthermore, all minimal solutions can be identified by finding all minimal covers of a set using some subsets. However, the situation becomes more complex in the case of the symmetrized and supertropical semirings. Here, our immediate aim is to give a suitable candidate for the greatest modulus solution (which turns out to be almost as easy as in the tropical case) and identify a ``representative'' set of minimal modulus solutions, which is done algorithmically by extending the minimal covers for the associated tropical system.  

Note that some of the previous works have studied a related problem in the symmetrized setting, particularly systems of linear balances (where the exact equation is replaced by a more elaborate relation called ``balance''). Baccelli et al.~\cite{baccelli_book} provide foundational results on solving such systems with balances: in particular, a unique signed solution of such system can be found by the tropical analogue of Cramer rules. In the supertropical semiring, linear algebra has been developed by Izhakian and Rowen~\cite{izhakian_supertropical}, focusing on vector spaces and bases. However, they also introduce another specific relation to replace the precise equality. In the present paper we address the linear systems $A \otimes x = b$ with precise equalities, as these directly relate to important cryptographic applications, such as breaking variants of the Stickel protocol, where solving such systems allows the attacker to recover the shared secret key. In this relation, we would like to attribute the initial idea of this paper to the MSc dissertation of Elt~\cite{Sarah_elt}, in which several implementations of Stickel protocol over the symmetrized semiring were suggested and analyzed.   

This paper is organized as follows: Section 2 begins with preliminaries and basic definitions related to tropical and layered tropical algebra, and one-sided linear systems over the tropical semiring. In Section 3, we introduce the new methods for addressing the solvability of linear systems over the symmetrized and supertropical semirings.

\section{Preliminaries}
In this section, we present some of the standard and less standard definitions of tropical algebra and matrix algebra over semirings. We also review some well-established theories concerning the solvability of tropical linear systems. To underscore the relevance of the results in this paper, we then briefly discuss tropical cryptography as a key motivational application. Note that we use the standard notation $[m]=\{1,\cdots,m\}$ for most common index sets. We begin by presenting the standard definition of general semiring and tropical semiring as one of the most prominent examples.

\begin{definition}[Semiring]
A semiring is a non-empty set $S$ equipped with two binary operations $\oplus$ and $\otimes$, which satisfy the following properties:
\begin{itemize}
    \item $(S, \oplus)$ is a commutative monoid which means that it satisfies associativity, commutativity and existence of an additive identity element $\epsilon$.
    \item $(S, \otimes)$ is a monoid which means that it satisfies associativity and existence of multiplicative identity element $e$.
    \item In $(S, \oplus,\otimes)$ multiplication $\otimes$ distributes over addition $\oplus$.
    \item The additive identity $\epsilon$ satisfies the absorbing property, that is $\epsilon \otimes e=e \otimes \epsilon=\epsilon$.
\end{itemize}
\end{definition}

\begin{definition}[Tropical semiring]
The tropical semiring $\mathbb{R}_{\max }$ is defined by $\mathbb{R}_{\max }=(\mathbb{R} \cup\{-\infty\}, \oplus, \otimes),$ where the tropical addition $\oplus$ and the tropical multiplication $\otimes$ are respectively defined by $a \oplus b=\max \{a, b\}$ and $a \otimes b=a+b$ for all $a, b \in \mathbb{R}_{\max }$. Similarly, the tropical semiring over integers $\mathbb{Z}_{\max }$ is defined by $\mathbb{Z}_{\max }=(\mathbb{Z} \cup\{-\infty\}, \oplus, \otimes)$.
\end{definition}

The semiring operations can be extended to vectors and matrices to form matrix algebra over a general semiring $S$. In particular, the operation $A \otimes \alpha=\alpha \otimes A$, where $\alpha \in S, A \in S^{m \times n}$ and $A_{ij}$ denotes the $(i,j)$-entry of $A$ for $i\in [m]$ and $j\in [n]$, is defined by
$$
(A \otimes \alpha)_{i j}=(\alpha \otimes A)_{i j}=\alpha \otimes A_{i j} \quad \forall i \in [m]\text { and } \forall j \in [n].
$$
The addition $A \oplus B$ of two matrices $A \in S^{m \times n}$ and $B \in S^{m \times n}$, where $A_{i j}$ and $B_{i j}$ denote the $(i,j)$-entries of $A$ and $B$ for $i\in [m]$ and $j\in [n]$, is defined by
$$
(A \oplus B)_{i j}=A_{i j} \oplus B_{i j} \quad \forall i \in [m]\text { and } \forall j \in [n].
$$
The multiplication of two matrices is also similar to the ``traditional'' algebra. Namely, we define $A \otimes B$ for two matrices, where $A \in S^{m \times p}$ and $B \in S^{p \times n}$, as follows:
$$
(A \otimes B)_{i j}=\bigoplus_{k=1}^p A_{i k} \otimes B_{k j}
\quad \forall i \in [m]\text { and } \forall j \in [n].
$$

\begin{definition}[Matrix powers]
\label{powers def}
 For $A \in S^{n \times n}$, the $n$-th power of $A$ is denoted by $A^{\otimes n}$, and is equal to
$$
A^{\otimes n}=\underbrace{A \otimes A \otimes \cdots \otimes A}_{n \text { times }}
.$$
\end{definition}
By definition, power $0$ of any square matrix equals the identity matrix.
\begin{definition} [Identity matrix]
 The identity matrix $I \in S^{n \times n}$ is of the form $(I)_{i j}=\delta_{i j}$ where
$$
\delta_{i j}= \begin{cases} e  & \text { if } i=j, \\ \epsilon & \text { otherwise. }\end{cases}
$$
\end{definition}

We subsequently define the matrix polynomials over the symmetrized tropical semiring.
\begin{definition} [Matrix polynomials]
Matrix polynomial is a function of the form
$$
A \mapsto p(A)=\bigoplus_{k=0}^d a_k \otimes A^{\otimes k},
$$
where $A \in S^{n \times n}$, and $a_k \in S$.
\end{definition}

In this paper we will consider two layered extensions of tropical semiring, for which we give the following definition adapted from \cite{AkianGaubertGuterman2014}, Proposition-Definition 2.12.

\begin{definition}[Layered tropical semiring \cite{AkianGaubertGuterman2014}] 
Let $T$ be a semiring with zero $\szero$ and $\mathbb{R}_{\max}$ be the tropical semiring. Then, extension of $\mathbb{R}_{\max}$ by $T$ is the set
\begin{equation*}
T\ltimes\mathbb{R}_{\max }=(T\backslash\{\szero\}\times\mathbb{R}_{\max }\backslash\{-\infty\})\cup\{(\szero,-\infty)\},
\end{equation*}
equipped with the operations
\begin{equation*}
(a,b)\oplus (a',b')=
\begin{cases}
 (a+a', b) & \text{if $b=b'$},\\
 (a,b) & \text{if $b>b'$},\\
 (a',b') & \text {if $b'>b$}.
\end{cases}\quad \text{and}\quad (a,b)\otimes (a',b')=(a\cdot a', b\otimes b')
\end{equation*}
where $(a,b), (a',b')\in T\ltimes\mathbb{R}_{\max}$,  $(+,\cdot)$ denote the arithmetical 
operations in $T$ and $(\oplus,\otimes)$ denote the arithmetical operations in $\mathbb{R}_{\max}$. 
\end{definition}

\begin{definition}[Modulus (absolute value) in a layered tropical semiring  \cite{AkianGaubertGuterman2014}]

The modulus is the projection map
$$
|\cdot| : T \ltimes \mathbb{R}_{\max} \to \mathbb{R}_{\max}, \quad |(a,b)| = b.
$$
\end{definition}
The modulus is extended componentwise to vectors and matrices as follows. For a vector
$x=(x_1,\cdots,x_n)^{\mathsf T}\in (T \ltimes \mathbb R_{\max})^n$ , where ${}^{\mathsf T}$ denotes vector transposition, and a matrix
$A \in (T \ltimes \mathbb R_{\max})^{m\times n}$ define
$$
|x| = (\,|x_1|,\cdots,|x_n|\,)^{\mathsf T} \in\mathbb R_{\max}^n,\qquad
|A| = (\,|A_{ij}|\,)_{i,j}\in\mathbb R_{\max}^{m\times n}.
$$
For $u,v \in \mathbb{R}_{\max}^n$, we write $u \le v$ if
$u_i \le v_i$ for all $i$. For vectors $x,y \in (T \ltimes \mathbb{R}_{\max})^n$ we then use the notation $|x| \le |y|$ to mean the componentwise inequality between their moduli. In what follows we will compare various solutions of $A\otimes x=b$ and show how to find what we call the greatest modulus solution and the minimal modulus solutions in the case of symmetrized and supertropical semirings (introduced below). We will now give the rigorous definitions of greatest modulus solution and minimal modulus solution. We will then introduce the symmetrized and supertropical semirings following \cite{AkianGaubertGuterman2014}.

\begin{definition}[Greatest modulus solution]
A solution $y$ to $A\otimes x=b$ over the layered tropical semiring is called the greatest modulus solution if the following holds: for any solution $x$ to $A\otimes x=b$ we have  $|y_j| \geq |x_j|$ for all $j \in [n]$. 
\end{definition}

\begin{definition}[Minimal modulus solution]
Let $A \otimes x = b$ be a system over a layered tropical semiring, and let $\mathcal{S}$ denote the set of its solutions. A solution $d \in \mathcal{S}$ is called a \emph{minimal modulus solution} (minimal with respect to absolute value) if there is no other solution $x \in \mathcal{S}$ such that $|x| \le |d|$ and $|x| \neq |d|$. Equivalently, $d$ is minimal with respect to the partial order on $\mathcal{S}$ induced by the componentwise order on the modulus vectors.
\end{definition}

The first layered extension which we consider is called the symmetrized semiring and the second extension is the supertropical semiring. Both semirings were initially defined without using the layered semiring concept, but this concept provides a convenient common ground for both of them.  

\begin{definition}[Four-element symmetrized Boolean semiring \cite{AkianGaubertGuterman2014}]
Let $\mathbb{B}_s$ be the set
$$
\mathbb{B}_s=\{\varepsilon,\;0,\;\ominus 0,\;0^{\bullet}\}.
$$
Define the binary operations $\oplus$ (addition) and $\otimes$ (multiplication) on $\mathbb{B}_s$ by the tables below.
$$
\small
\begin{array}{c|cccc}
\oplus & \varepsilon & 0 & \ominus 0 & 0^{\bullet} \\ \hline
\varepsilon & \varepsilon & 0 & \ominus 0 & 0^{\bullet}\\
0 & 0 & 0 & 0^{\bullet} & 0^{\bullet}\\
\ominus 0 & \ominus 0 & 0^{\bullet} & \ominus 0 & 0^{\bullet}\\
0^{\bullet} & 0^{\bullet} & 0^{\bullet} & 0^{\bullet} & 0^{\bullet}
\end{array}
\qquad
\begin{array}{c|cccc}
\otimes & \varepsilon & 0 & \ominus 0 & 0^{\bullet} \\ \hline
\varepsilon & \varepsilon & \varepsilon & \varepsilon & \varepsilon\\
0 & \varepsilon & 0 & \ominus 0 & 0^{\bullet}\\
\ominus 0 & \varepsilon & \ominus 0 & 0 & 0^{\bullet}\\
0^{\bullet} & \varepsilon & 0^{\bullet} & 0^{\bullet} & 0^{\bullet}
\end{array}
\normalsize
$$
Then $(\mathbb{B}_s,\oplus,\otimes)$ is an idempotent semiring with additive identity $\varepsilon$ and multiplicative identity $0$.
\end{definition}

\begin{definition}[Symmetrized semiring \cite{AkianGaubertGuterman2014}]
The symmetrized semiring is the semiring $\mathbb{B}_s\ltimes\mathbb{R}_{\max}$. Excluding its zero element it is naturally split in three  parts of the form $\{0\}\times\mathbb{R}$, $\{\ominus 0\}\times \mathbb{R}$ and $\{0^{\bullet}\}\times\mathbb{R}$. The elements of $\{0\}\times\mathbb{R}$ and $\{\ominus 0\}\times \mathbb{R}$  are called signed and we will denote $a:=(0, a)$ and $\ominus a:= (\ominus 0, a)$ for any $a\in\mathbb{R}$. The elements of 
$\{0^{\bullet}\}\times\mathbb{R}$ are called balanced and we will denote $a^{\bullet}=(0^{\bullet},a)$ for $a\in\mathbb{R}$.   
\end{definition}

\begin{definition}[Supertropical semiring \cite{AkianGaubertGuterman2014}]
The Supertropical Semiring is the extension $\mathbb{N}_2\ltimes\mathbb{R}_{\max}$ where $\mathbb{N}_2$ denotes the quotient of the semiring $\mathbb{N}\cup \{0\}$ by the equivalence relation for which the equivalence classes are $\{0\}$, $\{1\}$ (further denoted by $\overline{1}$) and $\{n\in \mathbb{N}\colon n\geq 2\}$ (further denoted by $\overline{2}$).

The whole semiring $\mathbb{N}_2\ltimes\mathbb{R}_{\max}$ (excluding its zero element) is then naturally split in two parts of the form $\{\overline{1}\}\times\mathbb{R}$ and $\{\overline{2}\}\times \mathbb{R}$. The elements of 
$\{\overline{1}\}\times\mathbb{R}$ are called tangible and can be identified with the elements of $\mathbb{R}$ itself, while the elements of $\{\overline{2}\}\times\mathbb{R}$ are called ghosts and they will be distinguished by using a $\circ$ sign, i.e., $a^{\circ}:=(\overline{2},a)$ for $a\in\mathbb{R}$.  
\end{definition}

\if{
The tropical semiring can be extended to incorporate the concept of signs, leading to a broader structure. Specifically, the tropical semiring $\mathbb{R}_{\max}$ can be expanded to a larger set $\mathbb{S}$, where $\mathbb{R}_{\max}$ represents the positive part. This extension is analogous to the construction of the integers $\mathbb{Z}$ as an extension of the natural numbers $\mathbb{N}$ in classical algebra. In order to construct this extension, we firstly present the following definition.

\begin{definition}[Tropical algebra of pairs~\cite{Butkovi_book}]
    Let $\mathbb{R}_{\max }^2$ be the set of tropical pairs. Then, for $a,b \in \mathbb{R}_{\max }^2$, the two operations $(\oplus,\otimes)$ are defined as
\[
a \oplus b = (a_1, a_2) \oplus (b_1, b_2) = (a_1 \oplus b_1, a_2 \oplus b_2)
\]

\[
a \otimes b = (a_1, a_2) \otimes (b_1, b_2) = (a_1 \otimes b_1 \oplus a_2 \otimes b_2, a_1 \otimes b_2 \oplus a_2 \otimes b_1)
\]
we then define the minus sign as $\ominus a=(a_2,a_1)$, the absolute value of $a$ as $|a|=a_1 \oplus a_2$ and the balance operator as $a^{\bullet}=a \ominus a=(|a|,|a|)$. Clearly, these operators have the following properties:
\begin{enumerate}

\item $a^{\bullet}=(\ominus a)^{\bullet}$.
\item ${a^{\bullet}}^{\bullet}=a^{\bullet}$.
\item $a \otimes b^{\bullet}=(a \otimes b)^{\bullet}$.
\item $\ominus(\ominus a)=a$.
\item $\ominus(a \oplus b)=(\ominus a) \oplus(\ominus b)$.
\item $\ominus(a \otimes b)=(\ominus a) \otimes b$.

\end{enumerate}
These properties allow us to write $a \oplus(\ominus b)=a \ominus b$ as usual.
\end{definition}
We know present the equivalence relation on $\mathbb{R}_{\max}^2$
\[
(a_1, a_2) \mathcal{R} (b_1, b_2) \Leftrightarrow \begin{cases}
a_1 \oplus b_2 = a_2 \oplus b_1 & \text{if } a_1 \neq a_2, b_1 \neq b_2 \\
(a_1, a_2) = (b_1, b_2) & \text{otherwise.}
\end{cases}
\]

This equivalence relation is used to define symmetrized tropical algebra.
\begin{definition}[Symmetrized tropical semiring\label{def: sym semiring}~\cite{Butkovi_book}]
    The symmetrized tropica semiring $\mathbb{S}$ is the union of three equivalence classes $\mathbb{S}^\oplus, \mathbb{S}^\ominus$ and $\mathbb{S}^\bullet$ where
    $$
    \begin{aligned}
&\mathbb{S}^\oplus= \overline{(t,-\infty)} = \{(t, c) \mid c < t\}, \\
&\mathbb{S}^\ominus= \overline{(-\infty, t)} = \{(c, t) \mid c < t\}, \\
&\mathbb{S}^\bullet= \overline{(t, t)} = \{(t, t)\} \quad \forall t,c \in \mathbb{R}_{\max}
\end{aligned}
$$
Thus, each element in the symmetrized tropical semiring represents an equivalence class, and an easier notation using a single number will be used to denote it.
\end{definition}
}\fi

\if{

In general, the operations over the symmetrized tropical algebra can be summarized as follows
$$
\begin{aligned}
& a \ominus b=a, \quad \text { if } |a|>|b| , \\
& b \ominus a=\ominus a, \quad \text { if } |a|>|b| \text {, } \\
& a \ominus a=a^\bullet . \\
&
\end{aligned}
$$

$a \in \mathbb{R}_{\max}$ (e.g. $|\ominus -8|=8 \in \mathbb{R}_{\max}$).\\
}\fi

Let us now revisit the well-established theory regarding the solvability of tropical linear systems of equations of the shape
\begin{equation} \label{Ax=b}
    A \otimes x=b. 
\end{equation}
Below it can be assumed that $A \in \mathbb R_{\max}^{m \times n}, x \in \mathbb R_{\max}^{n}$ and $ b \in \mathbb R_{\max}^{m}$ or that 
$A \in \mathbb Z_{\max}^{m \times n}, x \in \mathbb Z_{\max}^{n}$ and $ b \in \mathbb Z_{\max}^{m}$ (as $\mathbb Z$ is closed under the tropical arithmetics). Throughout the paper we call the original layered linear system $A\otimes x = b$ the ``symmetrized" or ``supertropical" system (as appropriate), and we call the corresponding absolute value system $|A|\otimes z = |b|$ the ``tropical" system.

Below we will always assume that all components of $b$ are finite and there is a finite entry in each column of $A$. This is a reasonable assumption due to the following remark. 

\begin{remark}
The element $-\infty$ has the property ``$a\oplus b=-\infty \Leftrightarrow a=b=-\infty$'', and also  ``$a\otimes b=-\infty\Leftrightarrow a=-\infty$ or $b=-\infty$''.  This implies that if some components of $b$ are not finite then $A\otimes x=b$ can be easily reduced to the case where all of them are finite. Indeed, $b_i=-\infty$ implies that $x_j=-\infty$ whenever $A_{ij}\neq -\infty$. This means that not only the $i$th row of the system but also every column $j$ corresponding to such components of $x$ can be deleted. Furthermore, if a column of $A$ does not have finite entries, then the corresponding component of $x$ can take arbitrary values and hence such column of $A$ can be also deleted.  The same algebraic properties are true and the same deletion process works for the systems $A\otimes x=b$ over the symmetrized semiring $\mathbb {B}_s\ltimes \mathbb {R}_{\max}$ and the supertropical semiring $\mathbb {N}_2\ltimes\mathbb {R}_{\max}$
meaning that we can assume the finiteness of $b$ and existence of finite entry in every column of $A$ also when solving $A\otimes x=b$ over these semirings.
\end{remark}

The theory of $A\otimes x=b$ over various extensions of the tropical semiring naturally builds upon the theory of such systems over the tropical semiring itself. Firstly, we define the vector $\bar{x}=\left(\bar{x}_1, \cdots, \bar{x}_n\right)^{\mathsf T}$ where
\begin{equation}
\label{e:barx}
\bar{x}_j=-\left(\max _i\left(A_{i j}-b_i\right)\right) \quad \forall j \in [n].
\end{equation}

We then define the set $S_j(x)$ for each component $x_j$ of the vector $x$ as the set of indices of the satisfied equations in the system by the $j$'th component, formally the rows $i$ such that $x_j \otimes A_{ij}=b_i$ holds. That is
$$
S_j(x)=\left\{i \in [m]: x_j \otimes A_{ij} =b_i\right\} \quad \forall j \in [n].
$$
Note that $S_j(\bar x)$ is non-empty for every $j$ and it consists of the indices for which the maximum in~\eqref{e:barx} is attained.
\begin{definition}[Set cover and minimal set cover]
    Let $[m]$ be a finite set, and let $S_1, \cdots, S_n$ be subsets of $[m]$. We say that $S_1, \cdots, S_n$ form a \textit{cover} of $[m]$ if $\bigcup_{j \in [n]} S_j = [m]$. Furthermore, $S_1, \cdots, S_n$ form a \textit{minimal cover} if $\bigcup_{j \in [n]} S_j = [m]$ and, for any $k \in [n]$, $\bigcup_{j \in [n],j \neq k} S_j \neq [m]$, meaning that removing any subset $S_k$ from the cover would result in a union that no longer covers $[m]$.
\end{definition}

We now present how these definitions can be applied to describe the solvability and the complete solution set of the tropical linear system. 
As shown by Di Nola et al.~\cite{DiNola87}, the following proposition applies to general max-$T$ semirings with continuous $T$-norms. Although this result is well-known also in the case of max-plus semiring, we include a short proof for the readers' convenience. 

\begin{proposition}
If system~\eqref{Ax=b} is solvable, then it has a finite set of minimal solutions and just one maximal solution, which is the greatest solution $\bar{x}$. With the number of minimal solutions denoted by $r$, the whole solution set $H$ is represented as 
$$H=\bigcup_{i=1}^r \{x\colon d^{(i)}\leq x\leq \bar x\},$$
where $d^{(i)}$ denotes the $ith$ minimal solution.
\end{proposition}
\begin{proof} The fact that $x\leq \bar{x}$ for any solution $x$ can be found, e.g., in~\cite{Butkovi_book}, Theorem 3.1.1.

If $x$ is a minimal solution, then its components are either $\bar{x}_j$ or $-\infty$: indeed, if $x_j$ is finite and $x_j<\bar{x}_j$, then setting $x_j=-\infty$ does not violate any equation of $A\otimes x=b$ but contradicts the minimality of $x$. This also implies that the set of minimal solutions is finite. 

If $x$ is not a minimal solution, then there exists a solution $x'$ such that $x'\neq x$ and $x'\leq x$. For each component $x'_j$ where $x'_j<x_j$ we can assume without loss of generality that $x'_j=-\infty$ since we have $A_{ij}\otimes x'_j<b_i$ for all such $j$ and all $i$.  If $x'$ is not a minimal solution then we can apply the same argument repeatedly, until we obtain a minimal solution after a finite number of steps. 

Finally, let $x$ satisfy $d^{(i)}\leq x\leq\bar{x}$ for some minimal solution $d^{(i)}$. Then we have $A\otimes d^{(i)}\leq A\otimes x\leq A\otimes \bar{x}$ and $A\otimes d^{(i)}=A\otimes\bar{x}=b$ implying that also $A\otimes x=b$.
\end{proof}

Note that the system $A\otimes x=b$ is solvable if and only if $\bar{x}$ is a solution, which is equivalent to the sets $S_j(\bar{x})$ (for $j \in [n]$) covering the entire index set $[m]$ (e.g., \cite{Butkovi_book} Corollary 3.1.2). Furthermore, a vector $x$ solves the system if and only if $x \leq \bar{x}$ and the corresponding sets $S_j(x)$ (for $j \in [n]$) also cover $[m]$ (e.g. \cite{Butkovi_book}, Theorem 3.1.1). The minimal solutions arise naturally from the minimal subsets $K \subseteq [n]$ such that $\bigcup_{j \in K} S_j(\bar{x}) = [m]$: for each such $K$, the vector with $x_j = \bar{x}_j$ for $j \in K$ and $x_j = -\infty$ otherwise is a minimal solution of the system (e.g., \cite{Markovskii2004}, Corollary 1).

Therefore, the task of finding the complete solution set involves enumerating all possible minimal ``reductions'' of $\bar{x}$ where only the components of $\bar{x}$ that are necessary to provide for a cover of all rows are left finite and all other components are set to $-\infty$. For this we need to enumerate all minimal covers of $[m]$ by $S_j(\bar x)$ for all $j \in [n]$.

It is worth mentioning that the minimum set cover problem (finding a cover of smallest cardinality) is NP-hard~\cite{karp1972reducibility}, whereas finding one inclusion-minimal cover can be done in polynomial time by starting with any valid cover and iteratively removing redundant sets. The problem of enumerating all inclusion-minimal covers, which is mostly relevant here, is equivalent to the hypergraph transversal enumeration problem, for which quasi-polynomial time algorithms may exist~\cite{Blasius}, though the number of such covers can be exponential in the input size, making it computationally challenging in practice.\\

One of the well-known ideas in algebraic cryptography is that two parties, commonly called Alice and Bob, can exchange certain mathematical information and, working on this information, further create two secret keys, which ``magically'' coincide due to some commutativity properties that come from the mathematics being used by these parties. Then this common key can be used by the parties to encrypt and decrypt the messages being sent. Following this idea, Grigoriev and Shpilrain~\cite{GrigorievShpilrainStickel} observed that the Stickel protocol~\cite{Stickel}, formerly implemented using the ``traditional'' matrix algebra over rings and fields, also works over the tropical semiring where it is protected against the linear algebra attacks since almost all matrices in the tropical algebra are non-invertible. It can be also easily observed that the same protocol works over any semiring, as any two polynomials of the same square matrix commute over any semiring.

Despite this, an eavesdropper with access to the public data and the exchanged messages can reconstruct the shared secret key. Kotov and Ushakov~\cite{KU_paper} proposed an attack that works by enumerating the whole solution set of a one-sided tropical linear system derived from the publicly transmitted messages and identifying a special solution that enables recovery of the shared secret key.

A later observation in~\cite{Spanish_Paper} significantly strengthens this attack. It shows that the attacker does not need to recover a particular solution to the derived system. Instead, any solution is sufficient to recover the shared secret key.

This attack makes the tropical Stickel protocol very insecure and motivates the search for 1) other implementations of this protocol based on some commuting classes of matrices over the tropical semiring (and other semirings), 2) semirings over which the system $A\otimes x=b$ is not too easy to solve. As for 2), one natural idea is to consider various kinds of layered tropical semirings such as the symmetrized semiring and the supertropical semiring, and this has been one of the motivations for \cite{PonmaheshkumarMultiparty25} where a variant of the supertropical semiring was used as a platform for a multi-party extension of the Stickel protocol presented above. The application of the approach developed in the present paper to cryptanalyse \cite{PonmaheshkumarMultiparty25} is discussed in \cite{AS7}.\\

It is not clear, however, how a solution to $A\otimes x=b$ can be found over a general layered tropical semiring. An immediate idea is that we can first define the sets 
$$
S_j(\bar{z})=\left\{i \in [m]: \bar{z}_j \otimes |A_{ij}| =|b_i|\right\} \quad \forall j \in [n],
$$
where $\bar{z}$ is the greatest solution of the tropical system $|A|\otimes z=|b|$ and then define 
 $A'$ by 
\begin{equation}
\label{e:A'}
A'_{ij}=
\begin{cases}
    A_{ij}, &\text{if $i\in S_j(\bar{z})$},\\
    \szero, &\text{otherwise},
\end{cases}
\end{equation}
where $\szero$ denotes the zero of the semiring. Then we have the following obvious observation. (Throughout this paper, a system is called "solvable" if it has at least one solution over the underlying semiring.)
\begin{proposition}
$A\otimes x=b$ is solvable if and only if $A'\otimes x=b$ is solvable.    
\end{proposition}
\begin{proof}
If $x$ is a solution of $A\otimes x=b$ or $A'\otimes x=b$ then we have $|A_{ij}|\otimes |x_j|< |b_i|$ for any $i\notin S_j(\bar{z})$, so switching off such $A_{ij}$ to $\szero$ or keeping its finite value will not violate any of the inequalities of any of these systems.
\end{proof}

However, the case of supertropical semiring considered below shows that further reduction of $A'\otimes x=b$ to a one-sided linear system over the semiring $T$ may be not straightforward.



  


\section{Solving tropical linear system over symmetrized and supertropical semirings}

We firstly present some results on the greatest modulus solutions of $A\otimes x=b$, and then present algorithms that find minimal modulus solutions of this system.

\subsection{Finding the greatest modulus solution to $A \otimes x=b$}

In this section we discuss the algorithms for finding a solution of greatest modulus of one-sided system $A\otimes x=b$. In the case of symmetrized semiring such solution exists if and only if $A\otimes x=b$ is solvable, and in the case of supertropical semiring such solution exists only under certain conditions. 

The upcoming theories and discussions will be always based on $\mathbb{Z}_{\max}$ and its symmetrized and supertropical extensions. It is also feasible to introduce analogous concepts in the general case (i.e. over $\mathbb{R}_{\max}$), although some results require modification due to the lack of discreteness.

In this paper, we focus on the discrete semiring $\mathbb{Z}_{\max}$ because our motivating applications are in cryptography which most naturally works with integers or rational numbers. All the introduced concepts extend to $\mathbb{R}_{\max}$, but statements that exploit discreteness (for example, replacing $\bar x_j$ by its immediate predecessor $\bar x_j-1$, or equivalently considering the set $(-\infty,\bar x_j-1]$) must be reformulated. In $\mathbb{R}_{\max}$, where no immediate predecessor exists, such arguments and claims should be stated in terms of the open intervals $(-\infty,\bar x_j)$.\\

We now present Algorithm~\ref{alg: x_sym} for computing the candidate greatest solution to $A \otimes x=b$ over the symmetrized semiring and Algorithm~\ref{alg: x_sup} for computing the same in the supertropical case. We denote by sign($a$) the sign of the symmetrized element $a$, which can be positive or negative. Note that the symmetrized and supertropical semirings do not satisfy the conditions of Theorems 2 or 3 of~\cite{Spanish_Paper}, where a general method to find the greatest solution of a linear system over an additively idempotent semiring is given.
\begin{algorithm}[H]
\caption{Computing the candidate greatest modulus solution over symmetrized semiring \label{alg: x_sym}}
\begin{algorithmic}[1]
\State \textbf{Inputs:} The symmetrized system $A \otimes x=b$.
\State \textbf{Output:} The vector $\bar x_{\sym}$.
\vspace{0.5cm}
\State Let $ \text{Signed\_equations} = \{ i \in [m] : b_i \text{ is signed} \} $ and $ \text{Balanced\_equations} = \{ i \in [m] : b_i \text{ is balanced} \} $.
\vspace{0.5cm}
\State Find the tropical system $|A| \otimes z=|b|$ by taking the absolute value of the symmetrized system, and find the greatest solution $\bar x$ and the sets $S_j(\bar x)=\{i \in [m]: \bar x_j \otimes |A_{ij}|=|b_i|\}$ for all $j \in [n]$.
\vspace{0.5cm}
\State Find $\widetilde S_j=S_j(\bar x) \cap \text{Signed\_equations}$ for all $j \in [n]$.
\vspace{0.5cm}
\State For each $j \in [n]$, compute the components $\bar x_{\sym,j}$ of $\bar x_{\sym}$ as follows.
\vspace{0.5cm}
\If{$\widetilde S_j \neq \emptyset$}
    \If{$A_{ij}$ is not balanced, and $\operatorname{sign}(b_i) =  \operatorname{sign}(A_{ij})$ for all $i \in \widetilde S_j$}
        \State Set $\bar x_{\sym,j}=\bar x_j$.
    
    \ElsIf{$A_{ij}$ is not balanced, and $\operatorname{sign}(b_i) = \ominus \operatorname{sign}(A_{ij})$ for all $i \in \widetilde S_j$}
        \State Set $\bar x_{\sym,j}=\ominus \bar x_j$.
    \Else
        \State Set $\bar x_{\sym,j}=\bar x_j -1$. 
    \EndIf
\EndIf
\vspace{0.5cm}
\If{$\widetilde S_j = \emptyset$}
    \State Set $\bar x_{\sym,j}=\bar x_j^\bullet$.
\EndIf

\end{algorithmic}
\end{algorithm}

\begin{algorithm}[H]
\caption{Computing the candidate greatest modulus solution in the supertropical case \label{alg: x_sup}}
\begin{algorithmic}[1]
\State \textbf{Inputs:} The supertropical system $A \otimes x=b$.
\State \textbf{Output:} The candidate greatest solution  $\bar x_{\sup}$.
\vspace{0.5cm}
\State Let $ \text{Tangible\_equations} = \{ i \in [m] : b_i \text{ is tangible} \} $ and $ \text{Ghost\_equations} = \{ i \in [m] : b_i \text{ is ghost} \} $.
\vspace{0.5cm}
\State Find the tropical system $|A| \otimes z=|b|$ by taking the absolute value of the supertropical system, and find the greatest solution $\bar x$ and the sets $S_j(\bar x)=\{i \in [m]: \bar x_j \otimes |A_{ij}|=|b_i|\}$ for all $j \in [n]$.
\vspace{0.5cm}
\State Find $\widetilde S_j=S_j(\bar x) \cap \text{Tangible\_equations}$ for all $j \in [n]$.
\vspace{0.5cm}
\State For each $j \in [n]$, compute the components $\bar x_{\sup,j}$ of $\bar x_{\sup}$ as follows.
\vspace{0.5cm}
\If{$\widetilde S_j \neq \emptyset$}
    \If{$A_{ij}$ is not ghost for all $i \in \widetilde S_j$}
        \State Set $\bar x_{\sup,j}=\bar x_j$.
    \Else
        \State Set $\bar x_{\sup,j}=\bar x_j -1$.
    \EndIf
\EndIf
\vspace{0.5cm}
\If{$\widetilde S_j = \emptyset$}
    \State Set $\bar x_{\sup,j}=\bar x_j^\circ$.
\EndIf
\end{algorithmic}
\end{algorithm}

Algorithms~\ref{alg: x_sym} and~\ref{alg: x_sup} compute the candidate greatest modulus solutions for symmetrized and supertropical systems, respectively, in $O(mn)$ time, as they involve scanning the $m \times n$ matrix to compute the tropical greatest solution, identify equation types, form intersection sets of size up to $m$, and assign signs or ghost status via $O(m)$ checks per variable.

The following theorem states that the greatest modulus solution of the system $A\otimes x=b$ in the symmetrized case exists if and only if the system is solvable and that Algorithm~\ref{alg: x_sym} indeed provides such solution.

\begin{theorem}[Solvability and greatest solution of the symmetrized system]\label{theorem: sym greatest solution}
    The symmetrized system $A \otimes x=b$ is solvable if and only if $\bar x_{\sym}$ is the greatest modulus solution of this system. 
\end{theorem}

\begin{proof}
    Let $x$ be any solution to the symmetrized system. This implies $|x|$ is a solution to the tropical system $|A| \otimes y=|b|$, which means $|x|$ corresponds to a cover $K$ of $[m]$ using $S_j(\bar x)=S_j=\{i \in [m]: \bar x_j \otimes |A_{ij}|=|b_i|\}$ for all $j \in [n]$. Here $\bar x$ is the greatest solution of the tropical system. This implies $|x_j|=\bar x_j$ for $j \in K$ and $|x_j|< \bar x_j$ for $j \notin K$.\\\\
We will now show that $\bar x_{\sym}$ is a solution by replacing, if necessary, the components of $x$ by the components of $\bar x_{\sym}$  and making sure that $A\otimes x=b$ still holds after such replacement. 
    For components $x_j$ where $|x_j|=\bar x_j$, we have the following cases:
    \begin{itemize}
        \item If $\widetilde S_j \neq \emptyset$ and sign($b_i$) $=$ sign($A_{ij}$) for all $i \in \widetilde S_j$, then $x_j$ must be equal to $\bar x_j=\bar x_{\sym,j}$ since $x$ is a solution. Note that if $x_j$ is $\ominus \bar x_j$ or $\bar x_j^\bullet$ the associated signed equations will no longer be satisfied.
        \item If $\widetilde S_j \neq \emptyset$ and sign($b_i$) $= \ominus$ sign($A_{ij}$) for all $i \in \widetilde S_j$, then $x_j$ must be equal to $\ominus \bar x_j=\bar x_{\sym,j}$ for the same rationale.
        \item  If $\widetilde S_j=\emptyset$, then $x_j$ can be replaced by $\bar x_j^\bullet=\bar x_{\sym,j}$ and $x$ is still a solution, because we have $|x_j| \otimes |A_{ij}|=|b_i|$ for all $i \in S_j$, which means $\bar x_j^\bullet \otimes A_{ij}=b_i$ still holds for all $i \in S_j$ since $b_i$ is balanced for all $i \in S_j$, and the remaining equations $i \notin S_j$ are still satisfied since $|x_j| \otimes |A_{ij}|<|b_i|$.
    \end{itemize}
    Then, for components $x_j$ where $|x_j|<\bar x_j$, we have the following cases:
    \begin{itemize}
        \item If $\widetilde S_j \neq \emptyset$ and sign($b_i$) $=$ sign($A_{ij}$) for all $i \in \widetilde S_j$, then $x_j$ can be replaced with $\bar x_j=\bar x_{\sym,j}$ since $x$ is a solution. This is because we have $|x_j| \otimes |A_{ij}|< |b_i|$ for all $i \in [m]$, and if we do the replacement, the tropical system is still satisfied since we will have $|x_j| \otimes |A_{ij}|=|b_i|$ for all $i \in S_j$ and $|x_j| \otimes |A_{ij}|<|b_i|$ for all $i \notin S_j$. Also, the symmetrized system remains satisfied, as the signed equations are still satisfied since we used $\bar x_j$ with the appropriate sign, and the balanced equations are also preserved as adding a signed or balanced element doesn't change the sign of a balanced left-hand side.
        \item If $\widetilde S_j \neq \emptyset$ and sign($b_i$) $= \ominus$ sign($A_{ij}$) for all $i \in \widetilde S_j$, then $x_j$ can be replaced with $\ominus \bar x_j=\bar x_{\sym,j}$ due to the same argument.
        \item If $\widetilde S_j = \emptyset$, then $x_j$ can be replaced with $\bar x_j^\bullet=\bar x_{\sym,j}$ and $x$ is still a solution. This is because we have $|x_j| \otimes |A_{ij}|< |b_i|$ for all $i \in [m]$, and if we do the replacement, the tropical system is still satisfied since we will have $|x_j| \otimes |A_{ij}|=|b_i|$ for all $i \in S_j$ and $|x_j| \otimes |A_{ij}|<|b_i|$ for all $i \notin S_j$. Also, the symmetrized system is still satisfied as this component doesn't affect any signed equation since $|x_j| \otimes |A_{ij}|<|b_i|$ for all signed equations, and $x_j \otimes A_{ij}=b_i$ still holds for balanced equations as $|x_j| \otimes |A_{ij}|=|b_i|$ and changing $x_j$ with $\bar x_j^\bullet$ do not affect the balanced left-hand side.
        \item if $\widetilde{S}_j \neq \emptyset$ and neither the case that sign$(b_i)$ $=$ sign$(A_{ij})$ holds for all $i\in\widetilde{S}_j$, nor the case that sign$(b_i)$ $=\ominus$ sign$(A_{ij})$ holds for all $i\in\widetilde{S}_j$, then we can replace $x_j$ with $\bar x_j-1$ and $x$ is still a solution. This is because we have $|x_j| \otimes |A_{ij}|< |b_i|$ for all $i \in [m]$, and if we do the replacement, we still have $|x_j| \otimes |A_{ij}|< |b_i|$ for all $i \in [m]$ since $|x_j|<\bar x_j$ which means all the already satisfied equations are not effected.
    \end{itemize}
    Note that in all cases the inequalities $|x_j|\leq |\bar{x}_{\sym}|$ are easy to see. 
\end{proof}

The next example illustrates how Algorithm~\ref{alg: x_sym}
is applied to find a greatest modulus solution of $A\otimes x=b$ in the symmetrized case. 

\begin{example} [Solvability and greatest solution of the symmetrized system\label{ex:Solvability of the symmetrized system}]

$$
\left(\begin{array}{ccc}
1 & 3 & \ominus 4 \\
0 & 3 & 4 \\
2 & 0 & \ominus 0
\end{array}\right) \otimes\left(\begin{array}{l}
x_1 \\
x_2 \\
x_3
\end{array}\right)=\left(\begin{array}{c}
0^\bullet \\
\ominus 0 \\
0^\bullet
\end{array}\right).
$$
To determine whether the system is solvable, we firstly compute $\bar x_{\sym}$ using Algorithm~\ref{alg: x_sym} and then verify if $\bar x_{\sym}$ satisfies the system.
The algorithm produces $\bar x_{\sym}=(-2^\bullet ,\ominus-3, \ominus-4)^{\mathsf T}$, indicating that the system is solvable as $\bar x_{\sym}$ is a solution.

More precisely, the tropical greatest solution is $\bar x_j=-\max_i(|A_{ij}|-|b_i|)=-\max_i|A_{ij}|$, hence $\bar x=(-2,-3,-4)^{T}$. The sets $S_j(\bar{x})=\{\,i: |A_{ij}|+\bar x_j=0\,\}$ are $S_1(\bar{x})=\{3\}$, $S_2(\bar{x})=\{1,2\}$ and $S_3(\bar{x})=\{1,2\}$. Since only $b_2=\ominus0$ is signed, we have $\widetilde S_1=\emptyset$, $\widetilde S_2=\{2\}$ and $\widetilde S_3=\{2\}$. By Algorithm~\ref{alg: x_sym} we therefore set $\bar x_{\sym,1}=-2^{\bullet}$ and choose the signs of $\bar x_{\sym,2},\bar x_{\sym,3}$ so that row $2$ is respected, yielding $\bar x_{\sym}=(-2^{\bullet},\;\ominus-3,\;\ominus-4\bigr)^{\mathsf T}$.

\end{example}

For the case of the supertropical semiring, a solvable one-sided linear system has no greatest modulus solution in general, as the next example shows. 

\begin{example} [Solvability and non-existence of greatest modulus solution in the supertropical case]

$$
\left(\begin{array}{ccc}
2 & 3 &  4 \\
0 & 3 & 4 \\
2 & 0 & 0
\end{array}\right) \otimes\left(\begin{array}{l}
x_1 \\
x_2 \\
x_3
\end{array}\right)=\left(\begin{array}{c}
0^\circ \\
 0 \\
0^\circ
\end{array}\right).
$$
The system is solvable since $(-2^\circ, -4, -4)^{\mathsf T}$ is a valid solution. However, an attempt to generate a greatest solution using Algorithm~\ref{alg: x_sup} gives $\bar x_{\sup}=(-2^\circ, -3, -4)^{\mathsf T}$, which fails to satisfy the second equation because it would instead evaluate to $0^\circ$.

In more detail, Algorithm~\ref{alg: x_sup} first determines the tropical greatest solution by $\bar x_j=-\max_i(|A_{ij}|-|b_i|)=-\max_i|A_{ij}|$, which gives $\bar x_1=-\max(2,0,2)=-2$, $\bar x_2=-\max(3,3,0)=-3$ and $\bar x_3=-\max(4,4,0)=-4$. The sets $S_j(\bar x)=\{\,i: |A_{ij}|+\bar x_j=0\,\}$ are $S_1(\bar x)=\{1,3\}$, $S_2(\bar x)=\{1,2\}$ and $S_3(\bar x)=\{1,2\}$. Since the right-hand side has $b_2=0$ (tangible) while $b_1$ and $b_3$ are ghosts, Algorithm~\ref{alg: x_sup} assigns $\bar x_{\sup,1}=-2^{\circ}$ and keeps $\bar x_{\sup,2}=-3$ and $\bar x_{\sup,3}=-4$ tangible, yielding the candidate $\bar x_{\sup}=(-2^{\circ},-3,-4)^{\mathsf T}$. However, this vector does not satisfy the second equation: indeed, $3\otimes \bar x_2=3+(-3)=0$ and $4\otimes \bar x_3=4+(-4)=0$, so the maximum $0$ is attained twice by tangible terms and their sum is therefore $0^{\circ}$, whereas the right-hand side of the second equation is $0$.
\end{example}

However, we can modify this example in some ways to get a different situation in which a greatest modulus solution exists. 

\begin{example} [Existence of greatest modulus solution in the supertropical case]
The most obvious modification is to turn all components of the right hand side into ghosts: 
$$
\left(\begin{array}{ccc}
2 & 3 &  4 \\
0 & 3 & 4 \\
2 & 0 & 0
\end{array}\right) \otimes\left(\begin{array}{l}
x_1 \\
x_2 \\
x_3
\end{array}\right)=\left(\begin{array}{c}
0^\circ \\
 0^{\circ} \\
0^\circ
\end{array}\right).
$$
The candidate greatest modulus solution computed by Algorithm~\ref{alg: x_sup} is $(-2^\circ, -3^{\circ}, -4^{\circ})^{\mathsf T}$ and it satisfies the system. Yet another modification is 
$$
\left(\begin{array}{ccc}
2 & 3 &  4 \\
0 & 3 & 0 \\
2 & 0 & 0
\end{array}\right) \otimes\left(\begin{array}{l}
x_1 \\
x_2 \\
x_3
\end{array}\right)=\left(\begin{array}{c}
0^\circ \\
 0 \\
0^\circ
\end{array}\right).
$$
In this case  Algorithm~\ref{alg: x_sup} computes $(-2^\circ, -3, -4^{\circ})^{\mathsf T}$ and it also satisfies the system.
\end{example}
This raises a question about the conditions under which a greatest modulus solution to $A\otimes x=b$ exists in the supertropical case. Indeed, if there exist two distinct components $j$ and $j'$ such that $ \widetilde{S}_j \cap \widetilde{S}_{j'} \neq \emptyset$, then at least one tangible equation will be violated, and therefore, $ \bar x_{\sup} $ is not a solution. In fact, we need the following condition to be satisfied:
\begin{equation}
\label{e:tangible_condition}
\forall i \in \text{Tangible\_equations},\quad \left| \left\{ j \in [n] : i \in \widetilde{S}_j, \bar{x}_{\sup,j}=\bar{x}_j \right\} \right| = 1.
\end{equation}
Here, following the notation introduced in Algorithm~\ref{alg: x_sup}, Tangible\_equations denotes the set of $i$ such that $b_i$ is tangible. 

\begin{theorem}
\label{theorem: sup greatest solution}
Consider the system $A\otimes x=b$ over the supertropical semiring.  
\begin{itemize}
    \item[(i)] If $\bar x_{\sup}$ defined in Algorithm~\ref{alg: x_sup} is a solution then it is a greatest modulus solution of $A\otimes x=b$;
    \item[(ii)] $\bar x_{\sup}$ is a solution if and only if the system $A\otimes x=b$ is solvable and condition~\eqref{e:tangible_condition} holds. 
\end{itemize}
\end{theorem}

\begin{proof} (i):  Let $x$ be a solution of $A\otimes x=b$. We need to show $|x_j|\leq |\bar x_{\sup,j}|$ for all $j$. First observe that $|\bar x_{\sup,j}|=\bar x_j$ unless there are ghosts among $A_{ij}$ for $i\in \widetilde{S}_j$. If $|\bar x_{\sup,j}|=\bar x_j$, then $|x_j|\leq |\bar x_{\sup,j}|$ since $|x|$ is a solution of $|A|\otimes y=|b|$ and $\bar{x}$ is the greatest solution of this tropical system. If there are ghosts among $A_{ij}$ for $i\in\widetilde{S}_j$, then $|x_j|=\bar x_j$ is impossible and $\bar x_j-1=|\bar x_{\sup,j}|$ is the greatest modulus of this component. 

The above argument also gives us 
\begin{equation}
\label{e:ineq-useful}    
|x|\leq |\bar x_{\sup}|\leq \bar x,\quad \text{for}\; x\colon A\otimes x=b. 
\end{equation}

(ii): Suppose $ \bar x_{\sup} $ solves the system. Then $A\otimes x=b$ is solvable with $x=\bar x_{\sup}$. Furthermore, if a tangible equation index $i$ is contained in two distinct sets $\widetilde S_j $ and $ \widetilde S_{j'}$ for $j\neq j'$ and $\bar x_{\sup,j}=\bar x_j$ and $\bar x_{\sup,j'}=\bar{x}_j$ both hold then 
$$
      (A\otimes \bar x_{\sup})_i
      \;=\;
      A_{ij}\otimes\bar x_{\sup,j}
      \;\oplus\;
      A_{ij'}\otimes\bar x_{\sup,j'}.
$$
This supertropical tangible addition yields a ghost value  $ (b_i^\circ) $, violating this tangible equation, and hence contradicting that $ \bar x_{\sup} $ is a solution. Therefore, condition~\eqref{e:tangible_condition} should hold.

Conversely, assume that $x$ is a solution of $A\otimes x=b$. By \eqref{e:ineq-useful} $|x|\leq |\bar x_{\sup}|\leq \bar x$ implying that $|\bar x_{\sup}|$ solves $|A|\otimes y=|b|$. Thus we need to show that $\bar x_{\sup}$ satisfies each equation of the supertropical system $A\otimes x=b$, given that its modulus satisfies the tropical system $|A|\otimes y=|b|$.

First, each tangible $i$ lies in precisely one set $ \widetilde S_j$ for which $\bar{x}_{\sup,j}=\bar{x}_j$.  For that unique $j$, the definition of $\bar{x}_{\sup}$ ensures that
$$
      A_{ij}\otimes\bar x_{\sup,j} = b_i,
      \qquad
      A_{ik}\otimes\bar x_{\sup,k} < b_i
      \quad(\forall\,k\neq j).
$$
Thus, in any tangible equation index $i$, the maximum is uniquely attained at component $j$, giving the required tangible value $b_i$.  

\if{ 
Second, if $b_i$ is a ghost and $i\in S_j$ for some $j$ with $\widetilde{S}_j=\emptyset$, then the $i$th equation is satisfied since $A_{ij}\otimes \overline{x}_{\sup,j}=A_{ij}\otimes \overline{x}^{\circ}_{j}=b_i$.  
}\fi 

Second, if $b_i$ is a ghost and $|A_{ij}| \otimes |x_j|=|A_{ik}|\otimes |x_k|=|b_i|$ for some $j\neq k$, then by \eqref{e:ineq-useful} we also have 
 $|A_{ij}| \otimes |\bar x_{\sup,j}|=|A_{ik}|\otimes |\bar x_{\sup,k}|=|b_i|$ implying that the $i$th equation is satisfied by $\bar x_{\sup}$. 

Third, if $b_i$ is a ghost and $|A_{ij}| \otimes |x_j|=|b_i|$ with $A_{ij}$ being a ghost for some $j$, then by \eqref{e:ineq-useful} we also have 
$|A_{ij}| \otimes |\bar x_{\sup,j}|=|b_i|$ implying that the $i$th equation is satisfied by $\bar x_{\sup}$. 

Fourth, if $b_i$ is a ghost and $|A_{ij}|\otimes |x_j|=|b_i|$ with $x_j$ being a ghost for some $j$, then $x_j=\bar x_j^{\circ}$ and $\widetilde{S}_j=\emptyset$, otherwise $x$ would violate one of the tangible equations. In this case also $\bar x_{\sup,j}=x_j=\bar x_j^{\circ}$ thus $A_{ij} \otimes \bar x_{\sup,j}=b_i$ and $i$th equation of $A\otimes x=b$ is satisfied by $\bar x_{\sup}$.

The above case by case analysis shows that $A\otimes x=b$ is satisfied by $\bar x_{\sup}$, thus the proof is complete.  
\end{proof}

\subsection{Minimal modulus solutions}

We will now discuss a special case of $A\otimes x=b$ in which all entries of $b$ are tangible or signed.

\begin{proposition}
Consider the system $A\otimes x=b$ over the supertropical semiring. Assume that all entries of $b$ are tangible, condition~\eqref{e:tangible_condition} is satisfied, and that $S_j(|\bar x_{\sup}|)\neq\emptyset$ for all $j \in [n]$. Then the system has a unique solution, and this solution is $\bar x_{\sup}=\bar x$.
\end{proposition}

\begin{proof}
By Theorem~\ref{theorem: sup greatest solution} part (ii), under the assumption that condition~\eqref{e:tangible_condition} holds, $\bar x_{\sup}$ is a solution of the system whenever the system is solvable. Under the conditions of the proposition, $S_j(|\bar{x}_{\sup}|)=S_j(\bar{x})=\widetilde{S}_j$ for all $j \in [n]$,  $\bar x_{\sup}=\bar x$ and all sets $\widetilde{S}_j=S_j(\bar{x})$ are pairwise disjoint. 
In this case, if $x$ is a solution then, since the solution of the tropical system $|A|\otimes y=|b|$ is unique  and equal to $\bar x$ (see e.g. \cite{Butkovi_book} Theorem 3.1.6), we have $|x|=\bar x$. None of the components of $x$ can be ghosts since $S_j(\bar x)$ is not empty for all $j$ and a ghost component of $x$ would lead to one of the equations of $A\otimes x=b$ being violated. This implies that $x=\bar x$.

\if{
Condition~\eqref{e:tangible_condition} guarantees that each tangible equation $i$ belongs to exactly
one set $\widetilde S_j$. For any column $j$ with $\widetilde S_j\neq\emptyset$, there exists a tangible equation $i\in\widetilde S_j$ such that $A_{ij}\otimes x_j = b_i$. The same equality holds for $\bar x_{sup}$ by construction. Since $\bar x_{sup}$ is the maximal solution componentwise, we must have
$x_j=\bar x_{sup,j}$.

By assumption, $\widetilde S_j\neq\emptyset$ for all $j \in [n]$, hence the above argument applies to every component. Therefore $x=\bar x_{sup}$, proving that the solution is unique.
}\fi
\end{proof}

\begin{proposition}[\cite{Sarah_elt}\label{proposition:bsignec}] Consider the system $A\otimes x=b$ over symmetrized semiring, and let $b$ have only signed entries. Then, $A \otimes x=b$ is solvable if and only if there is a solution $x$ such that $|x|=z$ and $z$ is a minimal solution to $|A| \otimes z=|b|$.
\end{proposition}

\begin{proof}
``If'' part: obvious.

``Only if'' part: Suppose that $x$ is a solution of $A\otimes x=b$. 
We then have
$$
b_k=A_{k 1} \otimes x_1 \oplus A_{k 2} \otimes x_2 \oplus \cdots \oplus A_{k n} \otimes x_n=\bigoplus_{j \in [n]} A_{k j} \otimes x_j \quad \forall k \in [m].
$$
For any $k$, note that if for $j, l \in [n]$ we have $\left|A_{k j}\right| \otimes\left|x_j\right|>\left|A_{k l}\right| \otimes\left|x_l\right|$, then we may discard the term $A_{k l} \otimes x_l$ from the above summation.
Let the set $L_k \subset [n]$ be all such terms that are discarded for this reason in the summation for $b_k$. Observe that if $k \notin S_j(\bar{x})$ then $j \in L_k$. We have
$$
b_k=\bigoplus_{j \in [n] \backslash L_k} A_{k j} \otimes x_j .
$$

We now observe that all the terms now included in the summation should have equal absolute values and all have the same sign, the sign being the sign that $b_k$ has, and there should be no balance terms in the summation. Indeed, from the definition of addition in symmetrized tropical algebra, if we are summing equal values with different signs, or there are balance terms included in the summation, then the summation will produce a balance term. However, we have assumed that $b$ is signed so $b_k \neq \bigoplus_{j \in [n] \backslash L_k} A_{k j} \otimes x_j$ in this case. If the signs in the summation were the same, but different to the sign for $b_k$, the summation would then produce $\ominus b_k$ instead of $b_k$ and hence $b_k \neq \bigoplus_{j \in [n] \backslash L_k} A_{k j} \otimes x_j$. Therefore, we have that $A_{k i} \otimes x_i=A_{k j} \otimes x_j$ for all $i, j \in [n] \backslash L_k$.

By the above, all terms $A_{k j} \otimes x_j$ in the sum are such that $k \in S_j(\bar{x})$. If $|x|$ does not correspond to a minimal cover, there is an $S_j(\bar{x})$ which can be removed meaning that the corresponding term, $A_{k j} \otimes x_j$, can be removed, for each $j$ such that $k \in S_j(\bar{x})$. This means that $x_j$ can be set to $\szero$.

Gradually setting such redundant components to $\szero$ we obtain a solution to the symmetrized system, whose absolute value is a solution to the tropical system corresponding to a minimal cover.
\end{proof}

The following example shows that for a symmetrized system, $A \otimes x=b$, if $x$ is not a solution but is obtained from $z=|x|$, which is a solution to $|A| \otimes z=|b|$ and corresponding to the minimal cover, then we can not extend the minimal cover to find a solution that does satisfy the symmetrized system.

\begin{example}[\cite{Sarah_elt}] Consider the system $A\otimes x=b$ over the symmetrized semiring, where $A$ and $b$ are given by
$$
A=
\left(\begin{array}{ccc}
-2 & \ominus 2 & 2 \\
-5 & \ominus-3 & -2 \\
\szero & \szero & \ominus 3 \\
-3 & -3 & 2 \\
1 & \ominus 4 & \szero
\end{array}\right),\quad 
b=\left(\begin{array}{c}
\ominus 3 \\
\ominus-2 \\
1 \\
0 \\
\ominus 5
\end{array}\right).
$$
We observe that $\bar x=(3,1,-2)^{\mathsf T}$ is a solution to the tropical system $|A|\otimes z=|b|$ since $\bigcup_{j \in [3]} S_j(\bar{x})=[5]$. Indeed, we have $S_1(\bar{x}) \cup S_2(\bar{x}) \cup S_3(\bar{x})=[5]$ as $S_1(\bar{x})=\{2,4\}, S_2(\bar{x})=\{1,2,5\}$ and $S_3(\bar{x})=\{3,4\}$, which means that the tropical system is solvable. There is also a unique minimal cover given by $S_2(\bar{x}),S_3(\bar{x})$ and the corresponding solution of the tropical system is given by $z=(-\infty, 1,-2)^{\mathsf T}$. We try to reintroduce the signs to find a symmetrized solution $x$ such that $|x|=z$. However, we see that we are unable to choose signs for $x$ such that the third and fourth equations, where we have $b_3$ and $b_4$ on the right hand side, are both satisfied. This is because $b_3$ and $b_4$ both depend on the term $x_3$, but require $x_3$ to have different signs. Therefore, the minimal cover of the tropical system does not give a solution to the symmetrized system, hence $A \otimes x=b$ is unsolvable by Proposition~\ref{proposition:bsignec}.

Below, we will demonstrate that, indeed, extending the minimal cover of the tropical system will not provide a solution to the symmetrized system if the solution given by the minimal cover already does not provide a solution to the symmetrized system and $b$ is signed.

Let us extend the minimal cover to $S_1, S_2, S_3$, we would then have $|x|=(3,1,-2)^{\mathsf T}$. Now, observe that the entries $b_1, b_3$ and $b_5$ each only depend on one entry of $x$, which are $x_2$, $x_3$ and $x_2$ respectively. This means we must choose signs for $x_2$ and $x_3$ such that $b_1, b_3$ and $b_5$ are satisfied, as the signs we choose for the remaining entries of $x$ will not effect whether the signs for $(A \otimes x)_i=b_i, i=1,3,5$ are correct. In order to satisfy $b_1, b_3$ and $b_5$, we must have $x_2=1$ and $x_3=\ominus-2$.
The remaining question is, what sign do we choose for $x_1$? We have the following options:
\begin{enumerate}
    \item $x_1=3$, then we have $x=(3,1, \ominus-2)^{\mathsf T}$. This gives $A \otimes x=\left(\ominus 3,-2,1,0^\bullet, \ominus 5\right)^{\mathsf T} \neq b$.
    \item $x_1=\ominus 3$, then we have $x=(\ominus 3,1, \ominus-2)^{\mathsf T}$. This gives $A \otimes x=(\ominus 3, \ominus-2,1, \ominus 0, \ominus 5)^{\mathsf T} \neq b$.
\end{enumerate}
Therefore, extending our minimal cover does not provide a solution to the symmetrized system, as predicted by Proposition~\ref{proposition:bsignec}
\end{example}

We will now provide a counterexample that shows we indeed need $b$ to be signed in Proposition~\ref{proposition:bsignec}. This example also gives an idea how a minimal cover for the tropical system can be extended to yield a minimal modulus solution to the symmetrized system $A\otimes x=b$. The same idea will be used to construct minimal modulus solutions in the supertropical case.

\begin{example}[\cite{Sarah_elt} \label{ex: minimal cover}]
Consider the system $A \otimes x=b$ over the symmetrized semiring, where $A$ and $b$ are given by
$$
A=
\left(\begin{array}{ccc}
0 & \ominus-2 & 1 \\
0 & 3 & \ominus 0 \\
-5 & -3 & 1
\end{array}\right),\quad b=\left(\begin{array}{c}
4^{\bullet} \\
5 \\
\ominus 4
\end{array}\right).
$$

The greatest tropical solution is $\bar x=(4,2,3)^{\mathsf T}$ as $\bigcup_{j \in [3]} S_j(\bar{x})=[3]$ holds, where $S_1(\bar{x})=\{1\}, S_2(\bar{x})=\{2\}, S_3(\bar{x})=\{1,3\}$. The only minimal cover is $S_2(\bar{x}), S_3(\bar{x})$, however, the minimal cover does not provide a solution to the symmetrized system. The corresponding minimal solution to the tropical system $|A|\otimes z=|b|$ is $z=(-\infty, 2,3)^{\mathsf T}$. We would like to find a symmetrized solution $x$ such that $|x|=z$. Notice that $A_{11} \otimes x_1 \oplus A_{12} \otimes x_2 \oplus A_{13} \otimes x_3=A_{13} \otimes x_3$ and we need $A_{13} \otimes x_3=b_1$, therefore the entry $b_1$ only depends on $x_3$. We require $b_1=4^\bullet$ and have $A_{13}=1$, therefore we must have $x_3=3^\bullet$ so that $A_{13} \otimes x_3=1 \otimes 3=4^\bullet=b_1$. However, the entry $b_3$ only depends on $x_3$, i.e. we need $A_{33} \otimes x_3=b_3$. We require $b_3=\ominus 4$ and as $A_{33}=1$, we would need $x_3=\ominus 3$ to satisfy $b_3$. Since we cannot have $x_3=3^{\bullet}$ and $x_3=\ominus 3$, we cannot satisfy both $b_1$ and $b_3$, hence we cannot form a solution, $x$, to the symmetrized system from $|x|=(-\infty, 2,3)^{\mathsf T}$, which resulted from the minimal cover of the tropical system.
If Proposition~\ref{proposition:bsignec} could also hold for $b$ with some balanced components, there would be no solution for this system, therefore extending the minimal cover would not result in a solution. However, we will show that for this example, extending the minimal cover does provide us with a solution and therefore Proposition~\ref{proposition:bsignec} does not hold for $b$ not signed.

Let us extend the minimal cover to $S_1(\bar{x}), S_2(\bar{x}), S_3(\bar{x})$, then the solution to the tropical system is $z=(4,2,3)^{\mathsf T}$. We need to find a symmetrized solution $x$ such that $|x|=z$. Indeed $x=(4,2, \ominus 3)^{\mathsf T}$ satisfies the symmetrized system. Hence, we have found a symmetrized solution $x$ and $x$ has not been obtained from solution given by the minimal cover of the tropical system. Therefore, we have shown that the requirement that $b$ is signed is essential for Proposition~\ref{proposition:bsignec}.
\end{example}

Example \ref{ex: minimal cover} suggests that, to determine the whole solution set to the symmetrized system, we need to find all extended minimal covers of the tropical system. This involves finding every minimal cover and all its possible extensions. As we will show later, the solutions derived from these extended minimal covers are the minimal modulus solutions of the symmetrized system. By finding all such minimal modulus solutions, along with the greatest modulus solution, we can achieve a rough description of the entire solution set of the symmetrized system. See Algorithm~\ref{alg: minimal solutions symm} below. 

Although finding all minimal modulus solutions in the supertropical case does not imply that the whole solution set can be described using these solutions only, we also present a similar (and more easy) Algorithm~\ref{alg: minimal solutions super} and similar results for this case. 

\begin{algorithm}[H]
\caption{Computing the minimal modulus solutions of $A\otimes x=b$ (symmetrized) \label{alg: minimal solutions symm}}
\begin{algorithmic}[1]
\State \textbf{Inputs:} A solvable symmetrized system $A \otimes x=b$.
\State \textbf{Output:} The set of minimal modulus solutions.
\vspace{0.5cm}
\State Let $ \text{Signed\_equations} = \{ i \in [m] : b_i \text{ is signed} \} $ and $ \text{Balanced\_equations} = \{ i \in [m] : b_i \text{ is balanced} \} $.
\vspace{0.5cm}
\State Compute $\bar x_{\sym}$ using Algorithm~\ref{alg: x_sym}. 
\vspace{0.5cm}
\State Compute the sets $S_j(|\bar x_{\sym}|)=\left\{i \in [m]: |\bar x_{\sym,j}|\otimes\left|A_{i j}\right|=\left|b_i\right| \right\}$ for all $j \in [n]$.
\vspace{0.5cm}
\State Find all minimal covers of the tropical system. That is, find all minimal covers of $[m]$ using $S_j(|\bar x_{\sym}|)$.
\vspace{0.5cm}
\For {each minimal cover $K'$}
    \State Set $d_j=\bar x_{\sym,j}$ for all $j \in K'$ and $d_j=\szero$ for all $j \notin K'$.
    \State Check whether $d$ satisfies the symmetrized system, if so, append $d$ as a minimal solution.
    \State If not, find all minimal covers of the remaining tropical subsystem using some modified subsets $S_j^\checkmark$. That is, find all minimal covers of $\widetilde M=\{i \in Balanced\_equations:\bigoplus_{j=1}^{n} d_j \otimes A_{ij} \neq b_i\}$ using the sets $S_j^\checkmark$ for all $j \notin K'$, where \[
S_j^\checkmark = 
\left\{
\begin{array}{lll}
\{i \in \widetilde M :&|\bar x_{\sym,j}| \otimes |A_{ij}| = |b_i|\}, & \text{if } \widetilde S_j = \emptyset \\
\{i \in \widetilde M : &|\bar x_{\sym,j}| \otimes |A_{ij}| = |b_i|\ \text{and  $A_{ij}$  balanced or}& \\  &\bar x_{\sym,j} \otimes A_{ij} = \ominus \operatorname{sign}\left( \bigoplus_{l \in K'} d_l \otimes A_{il} \right) |b_i|\ \text{and}\ A_{ij}\ \text{signed}\}, & \text{if } \widetilde S_j \neq \emptyset.
\end{array}
\right.
\]

    \For{each minimal cover $K''$}
    \State Set $d_j=\bar x_{\sym,j}$ for all $j \in K''$.
    \EndFor
\EndFor
\State For each finite component $d_j$ where $S_j(|\bar x_{\sym}|)\neq \emptyset$ check if setting $d_j\to \szero$ produces a solution to the symmetrized system. If this is the case, discard $d$.\\ Discard all duplicate minimal solutions.
\end{algorithmic}
\end{algorithm}

Computing the greatest modulus solution and the sets $S_j(|\bar x_{\sym}|)$ requires $O(mn)$ time. The algorithm then enumerates all minimal covers of $[m]$ by these sets. Let $r_1$ denote the number of such covers. The enumeration of minimal covers is equivalent to the enumeration of hypergraph transversals, for which algorithms with quasi-polynomial output cost may exist~\cite{Blasius}. We therefore keep this cost abstract as $E(m,n)$. The total cost of this first enumeration is $O(E(m,n)\cdot r_1)$. For each minimal cover $K'$, the algorithm checks whether it produces a valid solution and constructs the subsystem (from the unsatisfied balanced equations), which requires $O(mn)$ time per cover, since Algorithm~\ref{alg: minimal solutions symm} additionally builds modified sets $S_j^\checkmark$. The algorithm then enumerates minimal covers of the subsystem: if $r_2(K')$ denotes their number, this second enumeration costs $O(E(m,n)\cdot r_2(K'))$, and verifying the corresponding candidate solutions costs $O(mn\cdot r_2(K'))$. Writing $R = r_1 + \sum_{K'} r_2(K')$ for the total number of minimal covers enumerated across both stages, the overall time complexity of Algorithm~\ref{alg: minimal solutions symm} is $O((E(m,n)+mn)\cdot R + mn)$.

\begin{algorithm}[H]
\caption{Computing the minimal modulus solutions of $A\otimes x=b$ (supertropical) \label{alg: minimal solutions super}}
\begin{algorithmic}[1]
\State \textbf{Inputs:} A supertropical system $A \otimes x=b$.
\State \textbf{Output:} The set of minimal modulus solutions.
\vspace{0.5cm}
\State Let $ \text{Tangible\_equations} = \{ i \in [m] : b_i \text{ is tangible} \} $ and $ \text{Ghost\_equations} = \{ i \in [m] : b_i \text{ is ghost} \} $.
\vspace{0.5cm}
\State Compute $\bar x_{\sup}$ using Algorithm~\ref{alg: x_sup}.
\vspace{0.5cm}
\State Compute the sets $S_j(|\bar x_{\sup}|)=\{i \in [m]: |\bar x_{\sup,j}| \otimes |A_{ij}|=|b_i| \}$ for all $j \in [n]$.
\vspace{0.5cm}
\State Find all minimal covers of the tropical system. That is, find all minimal covers of $[m]$ using $S_j(|\bar x_{\sup}|)$.
\vspace{0.5cm}
\For {each minimal cover $K'$}
    \State Set $d_j=\bar x_{\sup,j}$ for all $j \in K'$ and $d_j=\szero$ for all $j \notin K'$
    \State Check whether $d$ satisfies the supertropical system, if so, append $d$ as a minimal solution.
    \State If not, find all minimal covers of the remaining tropical subsystem. That is, find all minimal covers of $\widetilde M=\{i \in Ghost\_equations:\bigoplus_{j=1}^{n} d_j \otimes A_{ij} \neq b_i\}$   using $S_j(|\bar x_{\sup}|)$ for all $j \not \in K'$.
    \For{each minimal cover $K''$}
    \State Set $d_j=\bar x_{\sup,j}$ for all $j \in K''$.
    \State Check whether $d$ satisfies the supertropical system, if so, append $d$ as a minimal modulus solution.
    \EndFor
\State For each finite component $d_j$ where $S_j(|\bar x_{\sup}|)\neq \emptyset$ check if setting $d_j\to \szero$ produces a solution to the supertropical system. If this is the case, discard $d$.
\EndFor
\State Discard all duplicate minimal solutions.
\end{algorithmic}
\end{algorithm}

Algorithm~\ref{alg: minimal solutions super} follows the same two-stage structure as Algorithm~\ref{alg: minimal solutions symm}, but with a simpler second stage. As in Algorithm~\ref{alg: minimal solutions symm}, computing the greatest modulus solution and the sets $S_j(|\bar x_{\sup}|)$ requires $O(mn)$ time, and the first enumeration of minimal covers of the whole system produces $r_1$ covers at a cost of $O(E(m,n)\cdot r_1)$. For each such cover, the algorithm checks validity and identifies the subsystem in $O(mn)$ time. Unlike Algorithm~\ref{alg: minimal solutions symm}, no modified sets are constructed, and the second enumeration of minimal covers is performed directly on the subsystem using the original sets $S_j(|\bar x_{\sup}|)$. If $r_2(K')$ denotes the number of minimal covers of the subsystem associated with a first-stage cover $K'$, then the second enumeration and verification steps contribute $O(E(m,n)\cdot r_2(K') + mn\cdot r_2(K'))$. Letting again $R = r_1 + \sum_{K'} r_2(K')$ denote the total number of enumerated covers, the overall time complexity of Algorithm~\ref{alg: minimal solutions super} is $O((E(m,n)+mn)\cdot R + mn)$, with a smaller constant factor than in Algorithm~\ref{alg: minimal solutions symm} due to the absence of the modified-set construction.

We now give some examples demonstrating the work of Algorithms~\ref{alg: minimal solutions symm} and \ref{alg: minimal solutions super}.

\begin{example}[Minimal modulus solutions of the symmetrized system]
Find all minimal solutions of the symmetrized tropical system, $A \otimes x=b$, given by

$$
\left(\begin{array}{cccc}
16 & \ominus 15 & -5 & \ominus 17 \\
\ominus 17 & 16 & 4 & \ominus 18 \\
-6 & 1 & 10 & -3 \\
7 & -7 & 6 & 13 \\
\ominus 18 & \ominus 17 & 0 & 19 
\end{array}\right) \otimes\left(\begin{array}{l}
x_1 \\
x_2 \\
x_3 \\
x_4
\end{array}\right)=\left(\begin{array}{c}
6^{\bullet} \\
7^{\bullet} \\
\ominus 3 \\
2 \\
8
\end{array}\right).
$$
We have $\bar x_{\sym}=(\ominus -10,\ominus -9, \ominus -7, -11)^{\mathsf T}$ by Algorithm~\ref{alg: x_sym}. We then use the sets $S_1(|\bar x_{\sym}|)=\{1,2,5\}$, $S_2(|\bar x_{\sym}|)=\{1,2,5\}$, $S_3(|\bar x_{\sym}|)=\{3\}$ and $S_4(|\bar x_{\sym}|)=\{1,2,4,5\}$ to find all minimal covers of $[m]=[5]$. There is a unique minimal cover in this case, namely $(3, 4)$, which corresponds to the vector $d = (\szero, \szero, \ominus -7, -11)^{\mathsf T}$. This vector does not satisfy the symmetrized system, as the first and second equations remain unsatisfied. We then need to find all minimal covers of $\widetilde M=\{1,2\}$ using $S_1^\checkmark=\{2\}$ and $S_2^\checkmark=\{1\}$. There exists a single minimal cover, $(1, 2)$, and the corresponding extended vector is then $(\ominus -10, \ominus -9, \ominus -7, -11)^{\mathsf T}$. This vector is clearly a minimal solution, as discarding any component would no longer satisfy the symmetrized system. 

\end{example}

The next example shows how Algorithm~\ref{alg: minimal solutions super} finds all minimal solutions of the supertropical system.

\begin{example} [Minimal modulus solutions of the supertropical system]
Consider the system
$$
\left(\begin{array}{ccc}
4 & 5 &  6 \\
4 & 2 & 2 \\
2 & 5 & 6 \\
\end{array}\right) \otimes\left(\begin{array}{l}
x_1 \\
x_2 \\
x_3
\end{array}\right)=\left(\begin{array}{c}
0^\circ \\
 0^\circ \\
0
\end{array}\right).
$$
The candidate greatest modulus solution is $\bar x_{\sup}=(-4^{\circ},-5,-6)^{\mathsf T}$.
We then use the sets $S_1(|\bar x_{\sup}|)=\{1,2\}$, $S_2(|\bar x_{\sup}|)=\{1,3\}$, and $S_3(|\bar x_{\sup}|)=\{1,3\}$ to find all minimal covers of $[m]=[3]$. There are two minimal covers, $(1, 2)$ and $(1, 3)$, which correspond to the vectors $d_1 = (-4^\circ,-5, \szero)^{\mathsf T}$ and $d_2 = (-4^\circ, \szero,-6)^{\mathsf T}$ respectively. These vectors are solutions and their minimality can be easily verified, since discarding any component would no longer satisfy the supertropical system. Therefore, the system is clearly solvable, as Algorithm~\ref{alg: minimal solutions super} produces minimal solutions.
\end{example}

The following example shows the insolvability of the supertropical system, since Algorithm~\ref{alg: minimal solutions super} returns no output.

\begin{example}[Insolvability of the supertropical system]
Consider the system
$$
\left(\begin{array}{cccc}
16 &  15 & -5 &  17 \\
 17 & 16 & 4 &  18 \\
-6 & 1 & 10 & -3 \\
7 & -7 & 6 & 13 \\
 18 &  17 & 0 & 19 
\end{array}\right) \otimes\left(\begin{array}{l}
x_1 \\
x_2 \\
x_3 \\
x_4
\end{array}\right)=\left(\begin{array}{c}
6^{\circ} \\
7^{\circ} \\
 3 \\
2 \\
8
\end{array}\right).
$$
The candidate greatest modulus solution is $\bar x_{\sup}=(-10,-9,-7,-11)^{\mathsf T}$.
We use the sets $S_1(|\bar x_{\sup}|)=\{1,2,5\}$, $S_2(|\bar x_{\sup}|)=\{1,2,5\}$, $S_3(|\bar x_{\sup}|)=\{3\}$ and $S_4(|\bar x_{\sup}|)=\{1,2,4,5\}$ to find all minimal covers of $[m]=[5]$. There is a unique minimal cover in this case, namely $(3, 4)$, which corresponds to the vector $d = (\szero, \szero, - 7, -11)^{\mathsf T}$. This vector does not satisfy the supertropical system, as the first and second equations remain unsatisfied. We then need to find all minimal covers of $\widetilde M=\{1,2\}$ using $S_1(|\bar x_{\sup}|)=\{1,2,5\}$ and $S_2(|\bar x_{\sup}|)=\{1,2,5\}$. There exist two minimal covers, $(1)$ and $(2)$, and the corresponding extended vectors are then $(-10, \szero, - 7, -11)^{\mathsf T}$ and $(\szero, -9, - 7, -11)^{\mathsf T}$ respectively. These vectors do not satisfy the supertropical system, demonstrating that the supertropical system is unsolvable.
\end{example}

Before the main work to prove the validity of Algorithms~\ref{alg: minimal solutions symm} and \ref{alg: minimal solutions super} we give some argument for the minimality criterion used in the end of these algorithms to eliminate non-minimal solutions.
Recall that $\widetilde S_j$ denotes the intersection of $S_j(\bar x)$ with the set of all indices of signed or tangible components of $b$.

\begin{proposition}\label{prop:minimal version}
Let $x$ be a solution to system $A\otimes x=b$ over symmetrized or, respectively, supertropical semiring and suppose that, for any $j\in [n]$, if $\widetilde S_j=\emptyset$ then either $|x_j|<\bar {x}_j$ or $x_j$ is balanced or, respectively, a ghost.  Under this condition, if $x$ is not a minimal modulus solution, then there is a solution $x'$ for which $x'_k=x_k$ for all but one component $k$ for which $x'_k=\szero\neq x_k$.
\end{proposition}
\begin{proof}
    If $x$ is not a minimal modulus solution, then there exists a solution $y$ such that $y\neq x$ and $|y|\leq |x|$. For each component $y_j$ where $|y_j|<|x_j|$ we can assume without loss of generality that $y_j=\szero$ since we have $|A_{ij}|\otimes |y_j|<|b_i|$ for all such $j$ and all $i$.  Let us also define $z$ by 
    \begin{equation}
    \label{eq:z}
       z_j=
       \begin{cases}
           \szero, & \text{if $y_j=\szero$},\\
           x_j, &\text{otherwise}.
       \end{cases}
    \end{equation}
    We will now show that, both in the symmetrized and in the supertropical case we can restore all but one of the components of $z$ to equal the corresponding components of $x$ without violating any equation of the symmetrized/supertropical system.


In the case of {\bf supertropical semiring}, we will consider the following cases:\\
    \textbf{Case 1:} tangible $b_i$: there is only one term  $A_{ij}\otimes x_j$ where $|A_{ij}\otimes  x_j|=|b_i|$, and for this term we necessarily have $A_{ij}\otimes y_j=A_{ij} \otimes z_j=b_i$, with it also being the only term where $|A_{ij}\otimes  z_j|=|b_i|$. This will not change if we restore some of the components of $z$ to $x_j$. \\
    \textbf{Case 2:} $b_i$ is a ghost, but there are no $j$ such that $A_{ij}\otimes y_j$ is a ghost and $|A_{ij}\otimes y_j|=|b_i|$. Then there should be at least two terms such that $|A_{ij}\otimes y_j|=|A_{ik}\otimes y_k|=|b_i|$ implying that the same holds for $z$ defined in~\eqref{eq:z}, so it satisfies the $i$th equation and this does not change if some components of $z$ are restored to $x_j$.\\
    \textbf{Case 3:} $b_i$ is a ghost and there is $j$ such that $A_{ij}\otimes y_j$ is a ghost and $|A_{ij}\otimes y_j|=|b_i|$. Then $A_{ij}\otimes x_j=A_{ij}\otimes z_j$ is also a ghost, since either $A_{ij}$ is a ghost or $y_j$ is a ghost implying that $\widetilde S_j=\emptyset$ and hence $x_j=z_j$ is also a ghost. Together with $|A_{ij}\otimes z_j|=|b_i|$ 
    this implies that $z$ defined in~\eqref{eq:z} satisfies the $i$th equation in this case and this does not change if some components of $z$ are restored to $x_j$.

In the case of {\bf symmetrized semiring}, we will consider the following similar cases:\\
    \textbf{Case 1:} signed $b_i$: for all indices $j$ such that $|A_{ij}\otimes  x_j|=|b_i|$, $A_{ij}\otimes x_j$ and $b_i$ have the same sign. Further, there are $j$ such that $|A_{ij}\otimes  y_j|=|b_i|$, and the signs of $A_{ij}\otimes y_j$ and $b_i$ are the same for all these terms, also implying $x_j=y_j=z_j$ for all such $j$. If we restore some components of $z$ to $x_j$, the signs of $A_{ij}\otimes x_j$ and $b_i$ will remain the same for all $j$ such that $|A_{ij}\otimes  z_j|=|b_i|$. \\
    \textbf{Case 2:} $b_i$ is balanced, but there are no $j$ such that $A_{ij}\otimes y_j$ is balanced and $|A_{ij}\otimes y_j|=|b_i|$. Then there should be at least two such terms $A_{ij}\otimes y_j$ and $A_{ik}\otimes y_k$ with opposite signs and moduli equal to $|b_i|$. Then $A_{ij}\otimes x_j$ and $A_{ik}\otimes x_k$ also have moduli equal to $|b_i|$. If $\widetilde S_j=\emptyset$ or $\widetilde S_k=\emptyset$ then $x_j=z_j$ or $x_k=z_k$ are balanced. If $\widetilde S_j\neq\emptyset$ and $\widetilde S_k\neq\emptyset$ then the signs of $y_j$ and $x_j=z_j$
    are the same and the signs of $y_k$ and $x_k=z_k$ are the same, determined by the signed equations whose indices appear in $\widetilde S_j$ and $\widetilde S_k$. In any case, $z$ satisfies the $i$th equation and this does not change if some components of $z$ are restored to $x_j$.  
    \\
    \textbf{Case 3:} $b_i$ is balanced and there is $j$ such that $A_{ij}\otimes y_j$ is balanced and $|A_{ij}\otimes y_j|=|b_i|$. The proof in this case follows that for Case 3 of the supertropical semiring case.

\end{proof}

The following example confirms the necessity of the above condition.

\begin{example}
Consider the system consisting of just one equation.
$$
0 \otimes x_1 \;\oplus\; (\ominus 0) \otimes x_2 \;=\; 0^\bullet.
$$
Take a solution 
$$
x^{(1)} = (0,0)^{\mathsf T}.
$$
For this example $\widetilde{S}_1=\widetilde{S}_2=\emptyset$ and none of the components of $x^{(1)}$ are balanced, which violates the conditions of Proposition~\ref{prop:minimal version}.
Note that no component can be removed from \(x^{(1)}\) while still preserving it as a solution. In this case the minimal modulus solutions are
$$
x^{(2)} = (0^\bullet, \szero)^{\mathsf T}, \quad x^{(3)} = (\szero, 0^\bullet)^{\mathsf T}.
$$
\end{example}

The next proposition confirms that if $A\otimes x=b$ is solvable, then minimal modulus solutions exist: this holds for any layered tropical semiring. 

\begin{proposition}\label{proposition: bounded down}
    Consider a system $A \otimes x=b$ over a layered tropical semiring. For any solution $x$ to this system there exists a minimal modulus solution $d$ such that $|d|\leq |x|$.
\end{proposition}
\begin{proof}
If $x$ is not a minimal modulus solution, then there exists a solution $x'$ such that $x'\neq x$ and $|x'|\leq |x|$. For each component $x'_j$ where $|x'_j|<|x_j|$ we can assume without loss of generality that $x'_j=\mathbf{0}$ since we have $|A_{ij}|\otimes |x'_j|<|b_i|$ for all such $j$ and all $i$.  If $x'$ is not a minimal modulus solution then we can apply the same argument repeatedly, until we obtain a minimal modulus solution after a finite number of steps.
\end{proof}

We next prove that Algorithm~\ref{alg: minimal solutions symm} and Algorithm~\ref{alg: minimal solutions super} find all minimal modulus solutions if the system $A\otimes x=b$ is solvable. 
\begin{theorem}\label{theorem: minimal solution has k'+k''}
    For any minimal solution $x$ of the symmetrized system there is a minimal solution $d$ found by Algorithm~\ref{alg: minimal solutions symm} such that $|x|=|d|$. In other words, any minimal solution $x$ corresponds to a cover of the form $K' \cup K''$, where $K'$ is a minimal cover of $[m]$ using $S_j(|\bar x_{\sym}|)$ and $K''$ is a minimal cover of $\widetilde M=\{i \in Balanced\_equations:\bigoplus_{j=1}^{n} d_j \otimes A_{ij} \neq b_i\}$ using $S_j^\checkmark$ for all $j \not \in K'$ (where $S_j^\checkmark$ is defined as in Algorithm~\ref{alg: minimal solutions symm}). 
\end{theorem}

\begin{proof}
    Our goal is to show that for any minimal modulus solution $x$, there exist a solution $d$ generated by Algorithm~\ref{alg: minimal solutions symm} such that $|d| \leq |x|$ and hence $|d|=|x|$. Since $x$ is a solution to the symmetrized system, this implies that $|x|$ is a solution to the tropical system $|A| \otimes y=|b|$, which means $|x|$ corresponds to a cover $K$ of $[m]$ constructed using $S_j(\bar x)=\{i \in [m]: \bar x_j \otimes |A_{ij}|=|b_i|\}$ excluding some $j \in [n]$ where $\widetilde S_j \neq \emptyset$ and for which either $A_{ij}$ is balanced for some $i \in \widetilde S_j$ or neither sign($b_i)=$ sign($A_{ij}$) holds for all $i\in \widetilde S_j$ nor sign($b_i)$ $=$ $\ominus$ sign($A_{ij}$) holds for all $i \in \widetilde S_j$. This is because for such components, we have 
    $|x_j| <\bar x_j$. If it were not the case, namely if $|x_j|=\bar x_j$ for such components, $x$ would not have been a solution.
    
    Therefore, this can be equivalently described by saying that the cover $K$ corresponding to $|x|$ is constructed using $S_j(|\bar x_{\sym}|)$, as $S_j(|\bar x_{\sym}|)=\emptyset$ precisely for $j$ where $|\bar x_{\sym,j}|=\bar x_j-1$ and these are the same components for which the above condition holds, that is, $\widetilde S_j \neq \emptyset$ and either there is a balanced $A_{ij}$ for some $i$ or neither sign($b_i)=$ sign($A_{ij}$) for all $i\in\widetilde S_j$ nor sign($b_i)$ $=$ $\ominus$ sign($A_{ij}$) for all $i \in \widetilde S_j$.
    
We also observe that $x_j=\bar{x}_{\sym,j}$ (both of them signed) for $j\in K$ and $\widetilde{S}_j\neq\emptyset$, while for $j\in K$ and $\widetilde{S}_j=\emptyset$
    we either also have $x_j=\bar{x}_{\sym,j}$ (both of them balanced) or
$|x_j|=|\bar{x}_{\sym,j}|$, $x_j$ signed and $\bar{x}_{\sym,j}$ balanced. In all of these cases, $|x_j|=|\bar x_{\sym,j}|=\bar x_j$.

Now, we can find a minimal cover $K'\subseteq K$ and following Algorithm~\ref{alg: minimal solutions symm} line 8 define $d$ by $d_j=\bar x_{\sym,j}$ for all $j\in K'$ and $d_j=\szero$ for $j\notin K'$. Then $d$ satisfies all signed equations due to the following argument. 


 Since $x$ is a solution, it satisfies $\bigoplus_{j \in K} x_j \otimes A_{ij}=b_i$ for every signed equation $i$. We can further simplify this expression by discarding the terms where $|x_j| \otimes |A_{ij}| < |b_i|$ yielding 
    \begin{equation}\label{eq:summation}
        \bigoplus_{\substack{j:j \in K, \hspace{0.05 cm} i \in \widetilde S_j}} x_j \otimes A_{ij} = b_i .
    \end{equation}
    In this refined summation, all remaining terms share the same absolute value and sign as $b_i$. Additionally, since $K'$ is also a cover of $[m]$ using $S_j(\bar x)$ and since $|d_j|=|x_j|$ for $j\in K'$, it follows that $|d|$ satisfies the tropical system. Therefore, $\bigoplus_{j \in K'} |d_j| \otimes |A_{ij}|=|b_i|$ is satisfied for every signed equation $i$, or equivalently 
    \begin{equation}
        \bigoplus_{\substack{j:j \in K', \hspace{0.05 cm} i \in\widetilde S_j}} |d_j| \otimes |A_{ij}| = |b_i| .
    \end{equation}
    after discarding all the terms with absolute values smaller than $|b_i|$. Since $K'$ is a subset of $K$ and since $d_j=x_j$ when $\widetilde S_j\neq \emptyset$ and $j\in K'$, the terms in this summation form a subset of those in summation~\eqref{eq:summation}, which implies $\bigoplus_{\substack{j: j \in K', \hspace{0.05 cm} i \in \widetilde S_j}} d_j \otimes A_{ij} = b_i$ is also satisfied. This implies that all signed equations are satisfied by $d$.\\

It is not necessarily guaranteed that the balanced equations are satisfied by $d$: in some of these equations the sum $\bigoplus_{j\in K'} d_j\otimes A_{ij}$ has the same absolute value as $b_i$, but it is signed while $b_i$ is balanced. The unsatisfied balanced equations form the set $\widetilde{M}$ defined in Algorithm \ref{alg: minimal solutions symm}. 

Let us prove that $\widetilde{M}$ is covered by sets $S_j^{\checkmark}$  for $j\in K\backslash K'$, where $S_j^{\checkmark}$ are defined in Algorithm \ref{alg: minimal solutions symm}. By contradiction, suppose that there is $i\in \widetilde{M}$ not covered by $S_j^{\checkmark}$ for all $j\in K\backslash K'$. But then, following the definition of these sets given in Algorithm~\ref{alg: minimal solutions symm} line 10, for any $j\in K\backslash K'$ we cannot have  $|x_j|\otimes |A_{ij}|=|b_i|$ when $\widetilde S_j=\emptyset$ or $\widetilde S_j\neq\emptyset$ but $A_{ij}$ is balanced (also since $|x_j|=|\bar x_{\sym,j}|$ for all $j\in K$), and neither we can have $x_j\otimes A_{ij}=\ominus\operatorname{sign} (\bigoplus_{l\in K'} x_l\otimes A_{il})|b_i|$ when $\widetilde S_j\neq\emptyset$ and $A_{ij}$ is signed: if this holds then, while $x_j\otimes A_{ij}=\bar x_{\sym,j}\otimes A_{ij}$ in this case, $\bigoplus_{l\in K'} \bar x_{\sym,l}\otimes A_{il}$ is different in sign and hence balanced, contradicting the definition of $\widetilde M$. The above analysis shows that $x$ cannot satisfy the $i$th equation of $A\otimes x=b$, a contradiction.\\

We now take a minimal cover $K''\subseteq K\backslash K'$ of $\widetilde M$ by the sets $S_j^{\checkmark}$ and define $d_j=\bar x_{\sym,j}$  for all $j\in K''$ following Algorithm \ref{alg: minimal solutions symm} line 12. 

The resulting (modified) vector $d$ is also a solution. Indeed, $|d|$ is a solution of the tropical system. Furthermore, signed equations are satisfied since $|d|\leq |x|$ and $x_j=\bar x_{\sym,j}=d_j$ for $j\in K'\cup K''$ and $\widetilde{S}_j\neq\emptyset$, and balanced equations with indices not in $\widetilde M$ are satisfied since $|d|\leq |x|$ and since they were already satisfied by $d$ before the modification, i.e., by $d$ that was defined using the cover $K'$. As for any equation with index $i\in \widetilde{M}$, following Algorithm~\ref{alg: minimal solutions symm} line 10 we see that either 1) there exists $j\in K''$ such that $d_j\otimes A_{ij}= b_i$ and either $\widetilde S_j=\emptyset$ (hence $d_j=\bar x_{\sym,j}$ is balanced) or $\widetilde S_j\neq\emptyset$ and $A_{ij}$ is balanced, or 2)
there exists $j\in K''$ with $\widetilde S_j\neq\emptyset$ and signed $A_{ij}$ such that $|d_j\otimes A_{ij}|=|b_i|$ and the sign of $d_j\otimes A_{ij}$ is opposite to the sign of $\bigoplus_{l\in K'} d_l\otimes A_{il}$. In both cases the left hand side of the $i$th equation of $A\otimes d=b$ is balanced meaning that it also holds.

It then follows that for any solution $x$, we have $|d| \leq |x|$ for some $d$ generated by Algorithm~\ref{alg: minimal solutions symm} before line 13. However, $x$ is a minimal modulus solution and hence we should have $|d|=|x|$, thus $d$ is also minimal and is not excluded in line 13 of Algorithm~\ref{alg: minimal solutions symm}.

\if{    Now, we can find a minimal cover $K'\subseteq K$ and define the corresponding vector $x'$ by $x'_j=x_j$ for $j \in K'$ and $x'_j=\szero$ for $j \notin K'$. This vector $x'$ satisfies all signed equations due to the following argument. Since $x$ is a solution, it satisfies $\bigoplus_{j \in K} x_j \otimes A_{ij}=b_i$ for every signed equation $i$. We can further simplify this expression by discarding the terms where $|x_j| \otimes |A_{ij}| < |b_i|$ yielding 
    \begin{equation}\label{eq:summation}
        \bigoplus_{\substack{j:j \in K, \hspace{0.05 cm} i \in S_j}} x_j \otimes A_{ij} = b_i .
    \end{equation}
    In this refined summation, all remaining terms share the same absolute value and sign as $b_i$. Additionally, since $K'$ is also a cover of $[m]$ using $S_j$, it follows that $|x'|$ satisfies the tropical system. Therefore, $\bigoplus_{j \in K'} |x'_j| \otimes |A_{ij}|=|b_i|$ is satisfied for every signed equation $i$, or equivalently 
    \begin{equation}
        \bigoplus_{\substack{j:j \in K', \hspace{0.05 cm} i \in S_j}} |x'_j| \otimes |A_{ij}| = |b_i| .
    \end{equation}
    after discarding all the terms with absolute values smaller than $|b_i|$. Since $K'$ is a subset of $K$, the terms in this summation form a subset of those in summation~\eqref{eq:summation}, which implies $\bigoplus_{\substack{j: j \in K', \hspace{0.05 cm} i \in S_j}} x'_j \otimes A_{ij} = b_i$ is also satisfied. This implies that all signed equations are satisfied.\\

It is not necessarily guaranteed that the balanced equations are satisfied by $x'$: in some of these equations the sum $\bigoplus_{j\in K'} x_j\otimes A_{ij}$ has the same absolute value as $b_i$, but it is signed while $b_i$ is balanced. The unsatisfied balanced equations form the set $\widetilde{M}$ defined in Algorithm \ref{alg: minimal solutions symm}. 

Let us prove that $\widetilde{M}$ is covered by sets $S_j^{\checkmark}$  for $j\in K\backslash K'$, where $S_j^{\checkmark}$ are defined in Algorithm \ref{alg: minimal solutions symm}. Indeed, if $|x|$ solves $|A|\otimes |x|=|b|$ and $K$, $K'$ and $\widetilde{M}$ are as defined above then $x$ solves the symmetrized system if and only if for each $i\in\widetilde{M}$ either 1) there exists $j\in K\backslash K'$ such that $|x_j\otimes A_{ij}|= |b_i|$ and either $A_{ij}$ is balanced or $x_j$ is balanced (but the latter is only possible when $\widetilde{S}_j=\emptyset$ as in the opposite case one of the signed equations would be violated), 2) there exists $k\in K\backslash K'$ such that $|x_k\otimes A_{ik}|= |b_i|$ and $x_k\otimes A_{ik}$ has the opposite sign with respect to $\bigoplus_{j\in K'} x_j\otimes A_{ij}$. 

We then see that for components $x_j$ for $j\in K\backslash K'$ with $\widetilde S_j\neq \emptyset$ the set of $i$ such that the terms $x_j\otimes A_{ij}$ have the sign described in case 2) above or are balanced is precisely $S_j^{\checkmark}$, and for components $x_j$ for $j\in K\backslash K'$ with $\widetilde S_j=\emptyset$ such set is, in general, a subset of $S_j^{\checkmark}$. Since such sets, when taken together, should cover $\widetilde M$ for $x$ to be a solution, the sets $S_j^{\checkmark}$ for $j\in K\backslash K'$ should cover $\widetilde M$ as well.\\ 

We now take a minimal cover $K''\subseteq K\backslash K'$ of $\widetilde M$ by the sets $S_j^{\checkmark}$ and define $d_j$ for all $j$ as in Algorithm \ref{alg: minimal solutions symm}. This is then also a solution. Indeed, all signed equations and balanced equations with indices not in $\widetilde{M}$ are satisfied since $|d|\leq |x|$ and $x'_j=x_{\sym,j}=d_j$ for $j\in K'\cup K''$ and $\widetilde{S}_j\neq\emptyset$,while for $j\in K'\cup K''$ and $\widetilde{S}_j=\emptyset$ we either have that $x'_j=x_{\sym,j}=d_j$ and $x'_j$ is balanced or that $|x'_j|=|x_{\sym,j}|=|d_j|$, $x'_j$ is signed and $d_j$ is balanced. As for the equations with indices in $\widetilde{M}$ one of the cases 1) or 2) holds for each $i\in\widetilde M$.  It then follows that for any solution $x$, we have $|d| \leq |x|$ for some $d$ generated by Algorithm~\ref{alg: minimal solutions symm} before line 13. However, $x$ is minimal and hence we should have $|d|=|x|$, thus $d$ is also minimal and is not excluded in line 13 of Algorithm~\ref{alg: minimal solutions symm}.
}\fi 

    \end{proof}

\begin{theorem}\label{theorem: minimal solution has k'+k'' supertropical}
    For any minimal solution $x$ of the supertropical system there is a minimal solution $d$ found by Algorithm~\ref{alg: minimal solutions super} such that $|x|=|d|$. In other words, any minimal solution $x$ corresponds to a cover of the form $K' \cup K''$, where $K'$ is a minimal cover of $[m]$ using $S_j(|\bar x_{\sup}|)$ and $K''$ is a minimal cover of $\widetilde M=\{i \in Ghost\_equations:\bigoplus_{j=1}^{n} d_j \otimes A_{ij} \neq b_i\}$   using $S_j(|\bar x_{\sup}|)$ for all $j \not \in K'$. 
\end{theorem}

\begin{proof}
    As in the proof of Theorem~\ref{theorem: minimal solution has k'+k''}, our goal is to show that for any minimal modulus solution $x$, there exist a solution $d$ generated by Algorithm~\ref{alg: minimal solutions super} such that $|d| \leq |x|$ and hence $|d|=|x|$. Since $x$ is a solution to the supertropical system, this implies that $|x|$ is a solution to the tropical system $|A| \otimes y=|b|$, which means $|x|$ corresponds to a cover $K$ of $[m]$ constructed using $S_j(\bar x)=\{i \in [m]: \bar x_j \otimes |A_{ij}|=|b_i|\}$ excluding some $j \in [n]$ where $\widetilde S_j \neq \emptyset$ and for which $A_{ij}$ is ghost for some $i \in \widetilde S_j$. This is because for such components, we have $S_j(x)=\emptyset$ as $|x_j| <\bar x_j$. If this is not the case, namely if $|x_j|=\bar x_j$ for such components, $x$ will not be a solution.\\
    Therefore, this can be equivalently described by saying that the cover $K$ corresponding to $|x|$ is constructed using $S_j(|\bar x_{\sup}|)$, as $S_j(|\bar x_{\sup}|)=\emptyset$ precisely for $j$ where $|\bar x_{\sup,j}|=\bar x_j-1$ and these are the same components for which the above condition holds, that is, $\widetilde S_j \neq \emptyset$ and there is a ghost $A_{ij}$ for some $i \in \widetilde S_j$.

    We also observe that $x_j=\bar{x}_{\sup,j}$ (both of them tangible) for $j\in K$ and $\widetilde{S}_j\neq\emptyset$, while for $j\in K$ and $\widetilde{S}_j=\emptyset$
    we either also have $x_j=\bar{x}_{\sup,j}$ (both of them ghosts) or
$|x_j|=|\bar{x}_{\sup,j}|$, $x_j$ is tangible and $\bar{x}_{\sup,j}$ is ghost. In all of these cases, $|x_j|=|x_{\sup,j}|=\bar x_j$.

    Now, we can find a minimal cover $K'\subseteq K$ of $[m]$ and define vector $d$ by $d_j=\bar x_{\sup,j}$ for $j\in K'$ and $d_j=\szero$ for $j\notin K'$ following Algorithm~\ref{alg: minimal solutions super}. 
    Then $|d|$ is a solution of the tropical system $|A|\otimes y=|b|$ since $K'$ is a cover, and $|d|\leq |x|$. Furthermore, $d$ satisfies all tangible equations due to the following argument. Take a tangible equation with index $i$. Since $x$ is a solution, there exists a unique component $j \in K$ such that $ |x_j| \otimes |A_{ij}|=|b_i|$ for every tangible equation $i$ and we also have $x_j\otimes A_{ij}=b_i$ and $\widetilde S_j\neq\emptyset$ for such terms. Since $K'$ is a subset of $K$ and a cover, we have $j\in K'$ and $d_j\otimes A_{ij}=b_i$ while $|d_k\otimes A_{ik}|<|b_i|$ for all $k\neq j$. \\
    
    It is not necessarily guaranteed that all ghost equations are satisfied by $d$. The unsatisfied ghost equations form the set $\widetilde{M}$ defined in Algorithm \ref{alg: minimal solutions super}. 
    
    Let us prove that $\widetilde{M}$ is covered by sets $S_j(|\bar x_{\sup}|)$  for $j\in K\backslash K'$. By contradiction, assume that some $i\in\widetilde M$ is not covered by these sets, then there is no $j\in K\backslash K'$ with $|\bar x_{\sup,j}\otimes A_{ij}|=|b_i|$ and hence no $j\in K\backslash K'$ with $|x_j\otimes A_{ij}|=|b_i|$, contradicting that $x$ is a solution: in this case in each equation of $A\otimes d=b$ with $i\in \widetilde M$ there is only one term with $|d_j\otimes A_{ij}| =|b_i|$ and hence in each equation of $A\otimes x=b$ there is only one term with $|x_j\otimes A_{ij}|=|b_i|$.  \\

    \if{
    Indeed, if $|x|$ solves $|A|\otimes |x|=|b|$ and $K$, $K'$ and $\widetilde{M}$ are as defined above then $x$ solves the supertropical system if and only if for each $i\in\widetilde{M}$, there exists $j$ such that $|x_j|\otimes |A_{ij}|= |b_i|$ and $\widetilde{S}_j=\emptyset$ as in the opposite case one of the tangible equations would be violated. For such components, the set of $i$ such that $|x_j| \otimes |A_{ij}|=|b_i|$ is precisely $S_j(|\bar x_{\sup}|)$. Since such sets, when taken together, should cover $\widetilde{M}$ for $x$ to be a solution, the sets $S_j(|\bar x_{\sup}|)$ for $j\in K\backslash K'$ should cover $\widetilde{M}$ as well.
    }\fi 
    
    We now take a minimal cover $K''\subseteq K\backslash K'$ of $\widetilde{M}$ by the sets $S_j(|\bar x_{\sup}|)$ and define $d_j=\bar x_{\sup,j}$ for $j\in K''$ as in Algorithm \ref{alg: minimal solutions super} line 12. 
    
    The resulting (modified) vector $d$ is also a solution. Indeed, $|d|$ is a solution of the tropical system. Furthermore, tangible equations are satisfied since $|d|\leq |x|$ and $x_j=\bar x_{\sup,j}=d_j$ for $j\in K'\cup K''$ and $\widetilde S_j\neq \emptyset$ and ghost equations with indices not in $\widetilde M$ are satisfied since $|d|\leq |x|$ since they were already satisfied by $d$ before the modification, i.e., by $d$ that was defined using the cover $K'$.  As for any equation with index $i\in \widetilde M$, it is satisfied by $d$ since there are at least two terms with absolute value $|b_i|$ one of them $A_{ij}\otimes d_j$ with $j\in K'$ and another $A_{ik}\otimes d_k$ with $k\in K''$. 

  Since $K' \cup K'' \subseteq K$, it follows that for any solution $x$, we have $|d| \leq |x|$ for some solution $d$ generated by Algorithm~\ref{alg: minimal solutions super} before line 13. However, $x$ is minimal and hence we should have $|d|=|x|$, thus $d$ is also minimal and is not excluded in line 14 of Algorithm~\ref{alg: minimal solutions super}.
 
    \end{proof}

\begin{corollary}
Every vector produced by Algorithm~\ref{alg: minimal solutions symm} up to line~12 is a solution of the symmetrized system.
\end{corollary}

\begin{proof}
Let $K'$ be the minimal cover chosen in Algorithm~\ref{alg: minimal solutions symm}. For any signed equation $i$, there exists $j\in K'$ with $|\bar x_{\sym,j}|\otimes|A_{ij}|=|b_i|$. By construction of $\bar x_{\sym}$ in Algorithm~\ref{alg: x_sym}, all dominant terms for signed equations have the same sign as $b_i$. Thus the nonempty sum over indices in $K'$ attains the correct modulus and sign, so every signed equation is satisfied. If a balanced equation $i$ is not satisfied by $K'$, then $i\in\widetilde M$. By definition of the sets $S_j^\checkmark$, the inclusion $i\in S_j^\checkmark$ ensures that adding component $j$ either contributes a balanced dominant term of modulus $|b_i|$ or a dominant term of opposite sign to the partial contribution from $K'$, making the total sum equal to $b_i$. Since $K''$ is chosen as a cover of $\widetilde M$ by the sets $S_j^\checkmark$, each 
$i\in\widetilde M$ has such a component $j\in K''$, and hence all remaining balanced equations are satisfied by extending $K'$ with $K''$. Therefore every vector constructed in line~12 of Algorithm~\ref{alg: minimal solutions symm} satisfies all equations, and is a solution.
\end{proof}

The above corollary does not hold in the supertropical case. Hence each candidate $d$ produced by Algorithm~\ref{alg: minimal solutions super} must be verified before accepting it as a minimal solution. The following example illustrates this necessity.

\begin{example}
    The supertropical system is given by
$$\begin{pmatrix}
0 & -1 & -1 \\
0 & 0 & -1
\end{pmatrix}
\otimes
\begin{pmatrix}
x_1 \\
x_2 \\
x_3
\end{pmatrix}
=
\begin{pmatrix}
2 \\
2^\circ
\end{pmatrix}.$$

Using Algorithm~\ref{alg: minimal solutions super}, the candidate greatest modulus solution is $\bar{x}_{\sup} = (2, 2^\circ, 3)^{\mathsf T}$. The sets $S_j(|\bar x_{\sup}|)$ are $S_1(|\bar x_{\sup}|) = \{1, 2\}$, $S_2(|\bar x_{\sup}|) = \{2\}$, and $S_3(|\bar x_{\sup}|) = \{1, 2\}$, and the minimal covers of $[m] = \{1, 2\}$ are $\{1\}$ and $\{3\}$.
For the minimal cover $K' = \{1\}$, the corresponding vector is $d = (2, \szero, \szero)^{\mathsf T}$, which does not satisfy the system. The remaining unsatisfied ghost equations are $\widetilde M = \{2\}$. The subsets for $j \notin K'$ (i.e., $j=2,3$) that cover $\widetilde M$ include $S_2(|\bar x_{\sup}|)$ and $S_3(|\bar x_{\sup}|)$. The minimal covers of $\widetilde M$ are thus $\{2\}$ and $\{3\}$. Extending $K' = \{1\}$ with $\{3\}$ yields $d = (2, \szero, 3)^{\mathsf T}$. However, this does not satisfy the system, as the first equation evaluates to $2^\circ$. Other extensions, such as $\{1, 2\}$ and $\{2, 3\}$, do yield solutions, but this shows that not every $d$ generated by Algorithm~\ref{alg: minimal solutions super} is a solution and must be checked.
\end{example}

\begin{corollary}
   Algorithms~\ref{alg: minimal solutions symm} and \ref{alg: minimal solutions super} generate only minimal modulus solutions, and for any minimal modulus solution $x$, they generate a minimal solution $d$ such that $|d|=|x|$.
\end{corollary}
\begin{proof}
The first part of the claim follows from Proposition~\ref{prop:minimal version} and the second part of the claim follows from Theorem~\ref{theorem: minimal solution has k'+k''} and Theorem~\ref{theorem: minimal solution has k'+k'' supertropical}.
\end{proof}

\begin{corollary}
The solutions generated by Algorithm~\ref{alg: minimal solutions symm} are not comparable, and so are the solutions generated by Algorithm~\ref{alg: minimal solutions super}.
\end{corollary}
\begin{proof}
This follows since all of them are minimal modulus solutions and we exclude all duplicate  solutions (in terms of modulus).    
\end{proof}

Note that the algorithms may generate duplicate minimal solutions before the elimination process. This occurs when two extended minimal covers derived from different minimal covers of $[m]$ generate the same solution. Additionally, the algorithms may produce non-minimal solutions before the minimality is checked by trying to set to $\mathbf{0}$ one of the components. After extending a minimal cover of $[m]$, it is possible that a subset of this minimal cover, along with the added components, will suffice to satisfy the system.\\


We now give some examples illustrating that there may be multiple minimal modulus solutions with the same modulus.

\begin{example} [Multiple minimal modulus solutions with the same absolute value in the symmetrized case]

$$
\left(\begin{array}{ccc}
1 & 3^\bullet & 4^\bullet \\
0 & 3 & 4 \\
2 & 0 & \ominus 4
\end{array}\right) \otimes\left(\begin{array}{l}
x_1 \\
x_2 \\
x_3
\end{array}\right)=\left(\begin{array}{c}
0^\bullet \\
\ominus 0 \\
0^\bullet
\end{array}\right).
$$

One of the minimal solutions generates by Algorithm~\ref{alg: minimal solutions symm} is $(-2^\bullet, \szero, \ominus -4)^{\mathsf T}$. However, another vector with the same absolute value that also satisfies the system is $(\ominus -2, \szero, \ominus -4)^{\mathsf T}$, illustrating that Algorithm~\ref{alg: minimal solutions symm} produces only one minimal solution from the set of minimal solutions sharing the same absolute value.
\end{example}

\begin{example}[Multiple minimal modulus solutions with the same absolute value in the supertropical case]

$$
\left(\begin{array}{ccc}
1 & 3^\circ & 4^\circ \\
0 & 3 & 4 \\
2 & 0 & 4
\end{array}\right) \otimes\left(\begin{array}{l}
x_1 \\
x_2 \\
x_3
\end{array}\right)=\left(\begin{array}{c}
0^\circ \\
0 \\
0^\circ
\end{array}\right).
$$

One of the minimal solutions generates by Algorithm~\ref{alg: minimal solutions super} is $(-2^\circ, \szero, -4)^{\mathsf T}$. However, another vector with the same absolute value that also satisfies the system is $(-2, \szero, -4)^{\mathsf T}$, illustrating that Algorithm~\ref{alg: minimal solutions super} produces only one minimal solution from the set of minimal solutions sharing the same absolute value.
\end{example}

We can conclude by the following theorem which gives a coarse description of the set of solutions to a symmetrized system.

\begin{theorem}
For any solution $x$, there exists $d$ generated by Algorithm~\ref{alg: minimal solutions symm} such that the absolute value $|x|$ satisfies $|d| \leq |x| \leq |\bar x_{\sym}|$.
\end{theorem}
\begin{proof}
    From Theorem~\ref{theorem: sym greatest solution}, we know that for any solution $x$, it holds that $|x| \leq |\bar x_{\sym}|$. Additionally, Theorem~\ref{theorem: minimal solution has k'+k''} implies that $|d| \leq |x|$ for one of $d$'s generated by Algorithm~\ref{alg: minimal solutions symm}. Combining these results, we conclude that $|d| \leq |x| \leq |\bar x_{\sym}|$ for one of such $d$'s.
\end{proof}

\begin{remark}
An analogue of this theorem holds in the supertropical case when we know that $\bar x_{\sup}$ is a solution. However, if it is not a solution then the structure of the solution set becomes more tricky and its investigation is left for future research.
\end{remark}

\section{Conclusion}
In this paper we have considered the systems of type $A\otimes x=b$ over the symmetrized and supertropical extensions of the tropical semiring, aiming to extend the well-known approach for solving the tropical systems consisting in 1) finding the greatest solution $\bar{x}$ of $A\otimes x=b$, 2) finding all minimal solutions of the same system by keeping the components of $\bar{x}$ that correspond to a minimal cover and switching off to $-\infty$ all other components. It turns out that this approach works well for the symmetrized semiring, but a slightly different greatest modulus solution $\bar{x}_{\sym}$ has to be computed (see Algorithm~\ref{alg: x_sym}). Furthermore, the minimal covers have to be extended. Then, the minimal modulus solutions can be found in a similar way by switching off the components whose indices are not taking part in the cover (see Algorithm~\ref{alg: minimal solutions symm}). For the supertropical semiring, we can similarly determine the ``candidate'' greatest solution $\bar{x}_{\sup}$ (see Algorithm~\ref{alg: x_sup}), which however may fail to be a solution, and then all minimal modulus solutions can be found similarly to the symmetrized case, by redefining the covering problem and then forming minimal modulus solutions after finding extended minimal covers (see Algorithm~\ref{alg: minimal solutions super}). 

We note that in both cases we are finding only one greatest modulus solution (or candidate solution) and only one minimal modulus solution for each extended cover. Therefore, there is a future research opportunity to, e.g., describe all minimal modulus solutions and the whole solution set more precisely. 

Regarding the perspectives of using the layered tropical semiring $T\ltimes\mathbb{R}_{\max}$ to implement the tropical Stickel protocol, it is clear that the system $A\otimes x=b$ over such semiring can be reduced to the system $A'\otimes x=b$ where $A'$ is defined by \eqref{e:A'}. However, further reduction of the system to a one-sided system over $T$ may be not straightforward, as the case of the supertropical semiring shows. This leaves a hope that using the layered tropical semiring $T\ltimes\mathbb{R}_{\max}$ instead of $T$ or $\mathbb{R}_{\max}$ could provide Alice and Bob with an extra layer of security (compared with a straightforward implementation using the semiring $T$ or $\mathbb{R}_{\max}$ only).

\section{Acknowledgement}
The authors are grateful to the referees for their careful reading, helpful suggestions and valuable comments.

\bibliography{references}

@book{Butkovi_book,
author = {P. Butkovi{\v{c}}},
title = {Max-linear {S}ystems: {T}heory and {A}lgorithms},
publisher={Springer},
address={London},
year={2010}
}

@article{DiNola87,
author={Di Nola, A. and Pedrycz, W. and Sessa, S.},
title={Fuzzy relation equations under {LSC} and {USC} $T$-norms and their {B}oolean solutions},
journal={Stochastica},
volume={11},
number={2-3},
year={1987}
}

@article{PonmaheshkumarMultiparty25,
author={Ponmaheshkumar, R. and Ramalingham, J. and Perumal, R.},
title={Multi-party key exchange scheme based on super-tropical semiring},
journal={Cryptologia},
url={https://doi.org/10.1080/01611194.2024.2435649},
year={2025},
note={published online}
}

@book{baccelli_book,
  title     = {Synchronization and Linearity: An Algebra for Discrete Event Systems},
  author    = {Baccelli, F.L. and Cohen, G. and Olsder, G.J. and Quadrat, J.P.},
  publisher = {John Wiley \& Sons Ltd},
  address   = {Chichester, UK},
  year      = {1992}
}

@INPROCEEDINGS{Stickel,
  author={Stickel, E.},
  booktitle={Third International Conference on Information Technology and Applications (ICITA'05)}, 
  title={A New Method for Exchanging Secret Keys}, 
  year={2005},
  volume={2},
  number={},
  pages={426-430},
  doi={10.1109/ICITA.2005.33}}

@article{KU_paper,
url = {https://doi.org/10.1515/jmc-2016-0064},
title = {Analysis of a key exchange protocol based on tropical matrix algebra},
author = {Kotov, M. and Ushakov, A.},
pages = {137--141},
volume = {12},
number = {3},
journal = {Journal of Mathematical Cryptology},
doi = {doi:10.1515/jmc-2016-0064},
year = {2018},
lastchecked = {2023-11-26}
}

@article{GrigorievShpilrainStickel,
  title={Tropical cryptography},
  author={Grigoriev, D. and Shpilrain, V.},
  journal={Communications in Algebra},
  year={2013},
  volume={42},
  pages={2624 - 2632},
  url={https://api.semanticscholar.org/CorpusID:6744219}
}

@Article{Spanish_Paper,
AUTHOR = {Otero Sánchez, Á. and Camazón Portela, D. and López-Ramos, J.A.},
TITLE = {On the Solutions of Linear Systems over Additively Idempotent Semirings},
JOURNAL = {Mathematics},
VOLUME = {12},
YEAR = {2024},
NUMBER = {18},
ARTICLE-NUMBER = {2904},
URL = {https://www.mdpi.com/2227-7390/12/18/2904},
ISSN = {2227-7390},
DOI = {10.3390/math12182904}
}

@article{Elbassioni,
title = {A note on systems with max–min and max-product constraints},
journal = {Fuzzy Sets and Systems},
volume = {159},
number = {17},
pages = {2272-2277},
year = {2008},
note = {Theme: Fuzzy Relations},
issn = {0165-0114},
doi = {https://doi.org/10.1016/j.fss.2008.02.020},
url = {https://www.sciencedirect.com/science/article/pii/S016501140800119X},
author = {K. Elbassioni}
}

@article{Blasius,
title = {Efficiently enumerating hitting sets of hypergraphs arising in data profiling},
journal = {Journal of Computer and System Sciences},
volume = {124},
pages = {192-213},
year = {2022},
issn = {0022-0000},
doi = {https://doi.org/10.1016/j.jcss.2021.10.002},
url = {https://www.sciencedirect.com/science/article/pii/S0022000021000994},
author = {T. Bläsius and T. Friedrich and J. Lischeid and K. Meeks and M. Schirneck}
}

@mastersthesis{Sarah_elt,
  author       = {Elt, S.},
  title        = {Cryptography in Symmetrised Tropical Algebra},
  school       = {University of Birmingham, School of Mathematics},
  address      = {Birmingham, UK},
  type         = {Master's thesis},
  year         = {2024},
  month        = mar
}

@article{izhakian_supertropical,
title={Supertropical linear algebra},
author={Z. Izhakian and L. Rowen},
journal={Pacific Journal of Mathematics},
volume={266},
number={1},
pages={43--75},
year={2013}
}

@incollection{AkianGaubertGuterman2014,
  TITLE = {{Tropical Cramer Determinants Revisited}},
  AUTHOR = {M. Akian  and S. Gaubert and A. Guterman},
  URL = {https://inria.hal.science/hal-00881203},
  NOTE = {See also arXiv:1309.6298},
  BOOKTITLE = {{Tropical and Idempotent Mathematics and Applications}},
  EDITOR = {G.L. Litvinov and S.N. Sergeev},
  PUBLISHER = {{AMS}},
  SERIES = {Contemporary Mathematics},
  VOLUME = {616},
  PAGES = {45},
  YEAR = {2014},
  HAL_ID = {hal-00881203},
  HAL_VERSION = {v1},
}

@PHDTHESIS{GaubertThesis, author={S. Gaubert}, title={Th\'eorie des syst\`emes lin\'eaires dans les dio\"\i des}, type={Th\`ese}, school={\'Ecole des Mines de Paris}, month={July}, year={1992}, }

@misc{AkianGaubertRowen2023,
      title={Linear algebra over semiring pairs}, 
      author={M. Akian and S. Gaubert and L. Rowen},
      year={2023-2025},
      note={Arxiv preprint 2310.05257},
      url={https://arxiv.org/abs/2310.05257}, 
}

@book{Cuninghame-Green1979,
author="Cuninghame-Green, R.A.",
title="Minimax algebra",
year="1979",
publisher="Springer",
address="Berlin, Heidelberg",
series="{L}ecture {N}otes in {E}conomics and {M}athematical {S}ystems",
volume={166},
isbn="978-3-642-48708-8",
doi="10.1007/978-3-642-48708-8_1",
url="https://doi.org/10.1007/978-3-642-48708-8_1"
}

@article{vorobyev1967extremal,
  title={Extremal algebra of positive matrices},
  author={Vorobyev, N.N.},
  journal={Elektron. Informationsverarbeitung und Kybernetik},
  volume={3},
  number={1},
  pages={39--72},
  year={1967},
note={in Russian}
}

@article{Markovskii2004,
  author    = {Markovskii, A. V.},
  title     = {Solution of Fuzzy Equations with Max-Product Composition in Inverse Control and Decision Making Problems},
  journal   = {Automation and Remote Control},
  volume    = {65},
  number    = {9},
  pages     = {1486--1495},
  year      = {2004},
  month     = {Sep},
  doi       = {10.1023/B:AURC.0000041426.51975.50},
  url       = {https://doi.org/10.1023/B:AURC.0000041426.51975.50},
  issn      = {1608-3032}
}

@Inbook{karp1972reducibility,
author="R. Karp",
editor="Miller, Raymond E.
and Thatcher, James W.
and Bohlinger, Jean D.",
title="Reducibility among Combinatorial Problems",
bookTitle="Complexity of Computer Computations: Proceedings of a symposium on the Complexity of Computer Computations, held March 20--22, 1972, at the IBM Thomas J. Watson Research Center, Yorktown Heights, New York, and sponsored by the Office of Naval Research, Mathematics Program, IBM World Trade Corporation, and the IBM Research Mathematical Sciences Department",
year="1972",
publisher="Springer US",
address="Boston, MA",
pages="85--103",
isbn="978-1-4684-2001-2",
doi="10.1007/978-1-4684-2001-2_9",
url="https://doi.org/10.1007/978-1-4684-2001-2_9"
}

@misc{AS7,
      author = {S. Alhussaini and S. Sergeev},
      title = {Cryptanalysis of Multi-Party Key Exchange Protocols over a Modified Supertropical Semiring},
      howpublished = {Cryptology {ePrint} Archive, Paper 2025/2062},
      year = {2025},
      url = {https://eprint.iacr.org/2025/2062}
}
\bibliographystyle{plain}

\end{document}